\documentclass[final,leqno,onefignum,onetabnum]{siamart171218}

\usepackage{amsmath}
\usepackage{amssymb}
\usepackage{latexsym}
\usepackage{srcltx}
\usepackage{graphicx}
\usepackage{xcolor}
\usepackage{amsfonts}
\usepackage{graphics}
\usepackage{tikz}
\usepackage[ruled,vlined]{algorithm2e}

\newcommand{\re}{{\mathbb R}}
\newcommand{\n}{{\mathbb N}}

\newcommand{\cA}{{\cal{A}}}

\newcommand{\cV}{{\cal{V}}}
\newcommand{\cL}{{\cal{L}}}

\newcommand{\cB}{{\cal{B}}}

\newcommand{\cT}{{\cal{T}}}
\newcommand{\cH}{{\cal{H}}}
\newcommand{\cS}{{\cal{S}}}
\newcommand{\cR}{{\cal{R}}}

\newcommand{\cC}{{\cal C}}

\newcommand{\cX}{{\cal{X}}}

\newcommand{\bT}{{\boldsymbol T}}

\newcommand{\balpha}{{\boldsymbol \alpha}}

\newcommand{\eps}{\varepsilon}

\newtheorem{prop}[theorem]{Proposition}
\newtheorem{cor}[theorem]{Corollary}
\newtheorem{remark}[theorem]{Remark}
\newtheorem{ex}[theorem]{Example}
\newtheorem{defi}[theorem]{Definition}

\def\eps{{\varepsilon}}

\title{Switching systems with dwell time:
computation of the maximal Lyapunov exponent\thanks{This work was
supported by  the Italian INdAM - G.N.C.S., and Russian RSF Grant 17-11-01027}} 

\author{
Yacine Chitour \thanks{Laboratoire des Signaux et Syst\`emes
Universit\'e Paris Saclay
91192 Gif-sur-Yvette, France.
(\email{Yacine.Chitour@l2s.centralesupelec.fr})}.
\and
Nicola Guglielmi\thanks{Gran Sasso Science Institute, via Crispi 7,
 I-$67100$ L'Aquila, Italy. 
(\email{nicola.guglielmi@gssi.it}).}
\and 
Vladimir Yu. Protasov\thanks{University of L'Aquila (Italy), 
Faculty of Computer Science of  National Research University Higher School of Economics 
(Moscow, Russia). (\email{v-protassov@yandex.ru}).}
\and 
Mario Sigalotti\thanks{Laboratoire Jacques-Louis Lions (LJLL)
Sorbonne Universit\'e
Paris, France (\email{Mario.Sigalotti@inria.fr}).}
}

\begin{document}
\maketitle
\newcommand{\slugmaster}{%
\slugger{siads}{xxxx}{xx}{x}{\small 1 December 2019}}
%
  
%
%

\begin{abstract}
We study asymptotic stability of continuous-time systems with mode-dependent guaranteed dwell time. These systems are reformulated as special cases of a general class of mixed (discrete-continuous) linear switching systems on graphs, 
in which some modes correspond to discrete actions and some others correspond to
continuous-time evolutions. Each discrete action has its own positive 
weight which accounts
for its time-duration. We develop a theory of stability for the mixed systems; in particular, 
we prove the existence of an invariant Lyapunov 
norm for mixed systems on graphs and study its structure in various cases, including  discrete-time systems for which discrete actions have inhomogeneous time durations. 
This allows us to adapt recent methods for the joint spectral radius 
computation (Gripenberg's algorithm and the Invariant Polytope 
Algorithm)  to  compute the Lyapunov exponent of mixed systems on graphs.  
%
%
\end{abstract}



\section{Introduction}
\label{sect:1}

Stability of continuous-time 
linear switching systems 
with fixed or guaranteed mode-dependent dwell time has generated a large amount of work in recent years, both from the theoretical and the numerical viewpoint, due to their widespread use in industry (see, for instance, \cite{Basso} as regards multilevel power converters and \cite{K2010} for on-line trajectory generation in robotics). 
 These systems represent an important class of hybrid dynamical systems, i.e., exhibit both continuous and discrete dynamic behavior \cite{Teel00}. 
  They usually consist of a finite number of subsystems and a discrete rule which dictates switching between them. From the theoretical perspective,
 the studies devoted to guaranteed (positive) dwell time started with the seminal works of Hespanha, Liberzon, and  Morse and  (see \cite{HespanhaMorse,LiberzonMorse,Morsed} and also 
 \cite{Liberzon}) and range from sufficient conditions for stability or stabilizability to $L_2$-stability \cite{CMS2}. Works considering guaranteed mode-dependent dwell time also provide sufficient conditions for stability or stabilizability in terms of LMIs or looped-functionals \cite{BS,GC}, which can be also extended to uncertain switching systems \cite{Xiang2,Xiang1}. More generally, when the dwell time is not fixed, the systems under consideration fall into the class of switching systems on non-uniform time domains or time scales (see \cite{Taousser}). On the numerical perspective for such issues, it turns out that the results of many of the previous works, since they deal with sufficient conditions for stability, also yield algorithms providing (only) upper bounds for the minimal dwell time insuring stability. These algorithms usually are based on LMIs or sum of squares programs \cite{chesi0}, but also on homogeneous rational Lyapunov functions \cite{Chesi1}. An important feature of such algorithms is that the provided upper bounds are guaranteed to converge to the minimal dwell time ensuring stability if one lets the degree of the approximating Lyapunov function (polynomial or rational) tend to infinity. 

In this paper, we propose new 
algorithms for computing the maximal Lyapunov exponent of continuous-time 
linear switching systems 
with fixed or guaranteed mode-dependent dwell time, which provide 
arbitrarily tight
upper and lower bounds for the minimal dwell time ensuring stability.
To this end we first introduce and 
analyse weighted discrete-time switching systems (Section~\ref{sec:2}), 
which are discrete-event switching systems with arbitrary switching for which the time-duration of every 
discrete event (the \emph{weight}) depends on the mode. 
Weights clearly affect the value of the joint spectral radius which can be associated with such systems. 
If all the matrices are nondegenerate, 
then a weighted system can be interpreted as 
a continuous-time switching system with fixed dwell times for each mode. 
Such systems behave similarly to classical discrete-time switching
systems (which correspond to the case of unit weights), but exhibit some significant distinctive feature. 
Nevertheless, we will see that 
the concept of invariant norm and the main algorithms of the  
joint spectral radius computation (Gripenberg's algorithm~\cite{G} and the invariant 
polytope algorithm~\cite{GWZ05,GP13, Mej}) can be extended to weighted systems after some 
modifications. 
  
Then in Section~\ref{sec:3} we introduce and study \emph{mixed systems}, 
which are a special class of hybrid systems: some modes correspond to 
discrete actions and some others correspond to continuous-time evolutions. With each discrete action is associated a weight which accounts for its time-duration. 
This type of systems with hybrid time domain is closely related to
another important class of switching systems, namely that of  impulsive switching systems \cite{Girard0,Xu1}. Again, when all discrete actions are nondegenerate, a mixed system can be interpreted as a continuous-time system with fixed dwell time for some 
modes and free dwell time for the others. 
We show in Theorem~\ref{th.40} how to extend to mixed systems existence results of extremal and invariant norms
known for discrete-time and continuous-time systems. 

Even mixed systems are not enough to tackle our original problem of efficiently computing Lyapunov exponents for continuous-time systems with guaranteed mode-dependent dwell time. 
In Section~\ref{sec:4} we make the next step by introducing constrained mixed systems 
or, more generally, mixed systems on graphs. They allow one to model 
constrains imposed on the 
order of activation of the modes 
along a trajectory. 
Such constraints on the  order are encoded in a multigraph~$G$. For classical discrete-time switching systems,
this generalization has been actively studied in the recent literature~\cite{CGP, D, K, PhJ1,  PhJ3, SFS} and such models are special occurrences of hybrid automata (cf.~\cite{VdSS}).
We extend the theory to the case including both discrete-time and continuos-time dynamics, introducing a rather general class of mixed systems on graphs, 
proving existence of extremal multinorms and characterizing them in terms of invariant polytopes. 
As a consequence, we extend the main algorithms for computing 
the Lyapunov exponent to such mixed systems on graphs. 

Then in Section~\ref{sec:5} we eventually address our main problem. We show that a continuous-time switching system with guaranteed mode-dependent dwell time can be seen as a special case 
of a mixed system on a certain graph. Using the techniques elaborated in 
Section~\ref{sec:4} we show how to decide the stability for those systems and how to  compute 
their Lyapunov exponents. Several illustrative examples are considered along with statistics of the efficiency of the  algorithms depending on the dimension of the system.

\section{Discrete-time weighted systems and continuous-time systems with fixed dwell times}\label{sec:2}

\subsection{Theoretical aspects}

We begin with the simplest case of restricted dwell times, 
when they are fixed for all modes, that is, we consider a continuous-time linear switching system $\dot x \, = \, B (t)x , \, t \in [0, +\infty)$, $x\in\mathbb{R}^d$, 
where at each time~$t$, the matrix $B(t)$ belongs to a finite family 
$\cB = \{B_1, \ldots , B_m\}$ of $d\times d$ matrices. 
We now impose the main restriction: 
each matrix~$B_j$ is associated with its fixed dwell time $\tau_j > 0$, and 
the switching law $B(\cdot)$ is a piecewise-constant function
such that each value $B_j$ is attained in a segment of length~$\tau_j$. 
In other words, every  matrix $B_j$ is switched on for a time exactly equal to $\tau_j$, 
after which it switches to another matrix $B_i$ for a time $\tau_i$
(the case $i=j$ is allowed), and so on. This can be actually seen as a discrete-time  switching system
$x(k) = A(k)x(k-1), \, k\in \n$, where the 
$d\times d$ matrices $A(k)$ are taken from the  set 
$\cA = \{A_1 = e^{\, \tau_1B_1}, \ldots, A_m = e^{\, \tau_mB_m}\}$. 
However, 
from the point of view of the rate of convergence or divergence of the system, 
by contrast with the classical 
framework of discrete-time switching systems where all modes are associated with a unit time duration, here the time duration of each action $A_j$ can be different. We now formally 
introduce the main concept.

Let  $\cA \, = \, \{A_1, \ldots , A_m\}$ be a family of $d\times d$-matrices 
(the \emph{modes}) and 
$\balpha = \{\alpha_1, \ldots , \alpha_m\}$ be a family of strictly  positive real numbers (the \emph{weights}). By $(\cA, \balpha)$ we denote the family of pairs $(A_j, \alpha_j)$, i.e.,  each matrix is equipped with its weight.  
\begin{defi}\label{d.10}
For a given family $(\cA, \balpha)$ as above, the corresponding \emph{weighted discrete-time switching system} (or, simply, \emph{weighted system}) 
is 
$$
x(k) \ = \ A(k)x(k-1)\, , \qquad A(k) \in \cA \, , k \in \n,
$$
 and the transfer from $x(k-1)$ to $x(k)$
takes time $\alpha (k)$, where $\alpha(k) \in \balpha$ is the weight of the 
matrix $A(k)$. 
\end{defi}
Thus, a classical discrete-time switching system is a weighted system with 
unit weight for each mode. 
The stability and asymptotic stability are defined in the same way 
as for the classical case. 
Note that  stability, asymptotic stability, and instability of a weighted system do not depend on the weights, as they are entirely defined by the boundedness of all trajectories
(or their convergence to zero, for the asymptotic stability). 
What really depends on weights is the rate of growth of the trajectory 
which is defined next.

\begin{defi}\label{d.20}
For a given weighted system $(\cA, \balpha)$, the \emph{$\balpha$-weighted joint spectral radius} (\emph{$\balpha$-spectral radius}, for short)
of $\cA$ 
is defined as 
\begin{equation}\label{eq.a-joint}
\rho (\cA, \balpha) \quad = \quad 
\limsup_{k \to \infty}\quad \max_{A(j) \in \cA,\, j = 1, \ldots , k}\quad  \bigl\|A(k)\cdots A(1)
\bigr\|^{\ \frac{1}{\alpha (k)+ \cdots + \alpha(1)}}.
\end{equation}
\end{defi}
Note that the definition of $\balpha$-spectral radius of $\cA$
does not depend on a specific norm on $\mathbb{R}^d$.
 
Let us first show that the $\limsup$ in \eqref{eq.a-joint} is actually a limit. 
This is the object the following result, which is based on a variant of Fekete's lemma presented in the appendix (Lemma~\ref{lem:fekete-convex}).

\begin{lemma}\label{lem:existence-lim}
Given a weighted system  $(\cA, \balpha)$, define 
  $$
  \rho_k= \max_{A(j) \in \cA,\, j = 1, \ldots , k}\quad  \bigl\|A(k)\cdots A(1)
\bigr\|^{\ \frac{1}{\alpha (k)+ \cdots + \alpha(1)}},
$$ for every $k\in\mathbb{N}$.
Then $\rho_k$ converges to $\inf_{j\in \mathbb{N}} \rho_j$.
\end{lemma}
\begin{proof}
For $j,k\in \mathbb{N}$, let $A(1),\dots,A(j+k)$ be such that 
\[ \rho_{j+k}=\bigl\|A(j+k)\cdots A(1)
\bigr\|^{\ \frac{1}{\alpha (j+k)+ \cdots + \alpha(1)}}.\]
Hence, 
\begin{align*}
\log(\rho_{j+k})&=\frac{1}{\alpha (j+k)+ \cdots + \alpha(1)}\log\left(\bigl\|A(j+k)\cdots A(1)
\bigr\|\right)\\
&\le \frac{1}{\alpha (j+k)+ \cdots + \alpha(1)}\left(\log\left(\bigl\|A(j+k)\cdots A(k+1)
\bigr\|\right)+\log\left(\bigl\|A(k)\cdots A(1)
\bigr\|\right)\right)\\
\le &\frac{\alpha (j+k)+ \cdots + \alpha(k+1)}{\alpha (j+k)+ \cdots + \alpha(1)}\log(\rho_{j})+\frac{\alpha (k)+ \cdots + \alpha(1)}{\alpha (j+k)+ \cdots + \alpha(1)}\log(\rho_{k}).
\end{align*}
As a consequence, for every $j,k\in \mathbb{N}$ there exists $\nu_{j,k}\in (0,1)$ such that 
\[\log(\rho_{j+k})\le \nu_{j,k}\log(\rho_{j})+(1-\nu_{j,k})\log(\rho_{k}),\]
with $\nu_{j,k}\le \frac{\alpha_{\max}}{\alpha_{\min}}\frac{j}{j+k}$. 
The conclusion then follows from Lemma~\ref{lem:fekete-convex} applied to $f(k)=\log(\rho_k)$. 
\end{proof}

In the classical case, when each $\alpha_j$ is equal to one, 
we keep the notation $\rho(\cA)$. The 
weighted joint spectral radius is equal to the biggest rate of asymptotic growth 
of trajectories. However, the weighted joint spectral radius 
is  not a positively homogeneous function of a matrix family $\cA$ as in the classical case. Instead, 
as stated in the next proposition, it
 is a homogeneous function of degree one with respect to $\balpha$.
\begin{prop}\label{p.10}
For an arbitrary weighted system $(\cA, \balpha)$ and for every $\lambda > 0$, 
we have 
\begin{equation}\label{eq.homog}
\rho \, \bigl(\delta_{\lambda}^{\balpha}(\cA)\, , \balpha \bigr) \quad = \quad  
 \lambda \, \rho \, \bigl(\cA\, , \balpha \bigr),
\end{equation}
where $(\delta_{\lambda}^{\balpha})_{\lambda>0}$ is the family  of dilations associated with $\balpha$, i.e., 
$$
\delta_{\lambda}^{\balpha}(A_1, \ldots , A_m)=(\lambda^{\alpha_1}A_1, \ldots , 
\lambda^{\alpha_m}A_m).
$$
\end{prop}
\begin{proof}
For every matrix product, we have 
$$
\Bigl[\, \lambda^{\alpha (k)}A(k) \cdots \lambda^{\alpha(1)}A (1)
\Bigl]^{\frac{1}{\alpha(k) + \cdots + \alpha(1)}}\quad = \quad 
\lambda\, \Bigl[\, A(k) \, \cdots \, A (1)
\Bigl]^{\frac{1}{\alpha(k) + \cdots + \alpha(1)}}\, . 
$$
Taking the maximum over all products of length $k$
 and the limit as $k\to \infty$, we deduce~(\ref{eq.homog}). 
 
\end{proof}

On the other hand, the sign of $\rho(\cA,\balpha)-1$ does not depend on $\balpha$, as 
stated in the next proposition.


\begin{prop}\label{p.20}
Let $(\cA, \balpha)$ be an arbitrary weighted system. 
Then 
$\rho (\cA, \balpha) = 1$ (respectively, $>1$ or $<1$) if and only if 
$\rho(\cA) = 1$ (respectively, $>1$ or $<1$). 
\end{prop}
\begin{proof}
Notice that
\[\left(\|A(k)\cdots A(1)\|^{\frac{1}{k}}\right)^{\frac{1}{\alpha_{\max}}}\le \|A(k)\cdots A(1)\|^{\frac{1}{\alpha(k) + \cdots + \alpha(k)}} \le \left(\|A(k)\cdots A(1)\|^{\frac{1}{k}}\right)^{\frac{1}{\alpha_{\min}}}\]
if $\|A(k)\cdots A(1)\|\ge 1$ and 
\[\left(\|A(k)\cdots A(1)\|^{\frac{1}{k}}\right)^{\frac{1}{\alpha_{\min}}}\le \|A(k)\cdots A(1)\|^{\frac{1}{\alpha(k) + \cdots + \alpha(k)}} \le \left(\|A(k)\cdots A(1)\|^{\frac{1}{k}}\right)^{\frac{1}{\alpha_{\max}}}\]
if $\|A(k)\cdots A(1)\|\le 1$.
In particular,
\begin{align*}
\rho (\cA)^{\frac{1}{\alpha_{\max}}}&\le \rho (\cA, \balpha)\le \rho (\cA)^{\frac{1}{\alpha_{\min}}},\quad&\mbox{if $\rho (\cA)>1$},\\
\rho (\cA)^{\frac{1}{\alpha_{\min}}}&\le \rho (\cA, \balpha)\le \rho (\cA)^{\frac{1}{\alpha_{\max}}},&\mbox{if $\rho (\cA)<1$},
\end{align*} 
and $\rho (\cA, \balpha) = 1$ if $\rho (\cA)=1$, proving the proposition.
%
\end{proof}

\begin{remark}\label{rmk:fenichel-weighted}
Proposition~\ref{p.20} 
allows one to generalize to weighted systems the following property of the joint spectral radius: a weighted system $(\cA, \balpha)$ is asymptotically stable if and only if $\rho (\cA, \balpha) < 1$. 
Indeed, as already noticed, $(\cA, \balpha)$ is asymptotically stable if and only if the discrete-time switching system associated with $\cA$ is asymptotically stable, which in turns happens if and only if $\rho(\cA)<1$ (see, for instance, \cite[Corollary 1.1]{J09}). 
\end{remark}

Combining Propositions~\ref{p.10} and \ref{p.20}, we obtain the following. 
\begin{theorem}\label{th.10}
The weighted joint spectral radius $\rho(\cA, \balpha)$ 
is equal to $\lambda^{-1}$, where $\lambda$ is found as the unique solution of the equation 
\begin{equation}\label{eq.wu}
\rho \  \bigl( \delta_\lambda^\balpha(\cA)
\bigr) \quad = \quad  1. 
\end{equation}
\end{theorem}

Equality~(\ref{eq.wu}) expresses implicitly the weighted joint spectral radius in terms of the 
joint spectral radius. The left hand side of~(\ref{eq.wu}) is an increasing 
function in $\lambda$, hence the root of this equation can be found merely by bisection in~$\lambda$. 
This, however, requires several computations of the joint spectral radius 
(of the family $\delta_\lambda^\balpha(\cA)$ 
for different values of $\lambda$). 
Therefore, it would be more efficient to compute the weighted joint spectral radius directly. The invariant polytope algorithm gives this opportunity. This issue is addressed at the end of this section.

Theorem~\ref{th.10} allows us to adapt many notions and results 
on the joint spectral radius to the weighted joint spectral radius. 
\begin{defi}\label{d.30}
A weighted system $(\cA,\balpha)$ is said to be 
\begin{itemize}
\item \emph{non-defective} if there exists a constant $C>0$ such that
\begin{equation}\label{eq.non-def}
\max_{A(j) \in \cA} \, \|A(k)\cdots A(1)\| \quad \le \quad C \, \rho^{\, \alpha(k) + 
\cdots + \alpha(1)},\qquad k \in \n\, ,  
\end{equation}
where $\rho = \rho(\cA, \balpha)$;
\item 
\emph{irreducible} if $\cA$ is irreducible, i.e., 
there exist no proper subspace $V\subset\mathbb{R}^d$ such that $A_jV\subset V$ 
for every $A_j\in\cA$.
\end{itemize}
\end{defi}

\begin{prop}\label{p.30}
Let 
$(\cA,\balpha)$ be a weighted system 
and $\rho = \rho(\cA, \balpha)$. 
If $
\delta_{1/\rho}^{\balpha}(\cA)$ 
is non-defective, 
then $(\cA,\balpha)$ also is.
\end{prop}
\begin{proof}
Let $\cA'=\delta_{1/\rho}^{\balpha}(\cA)$.
Since $\rho(\cA', \balpha)=\rho(\cA')=1$, the non-defectiveness of $\cA'$ reads 
%
 $$
 \max_{A'(j) \in \cA'} \ \|A'(k)\cdots A'(1)\| \quad \le \quad C\, \qquad k\in\mathbb{N}.
 $$
From the definition of $\cA'$, one recovers \eqref{eq.non-def}.
 
 \end{proof}

 As an immediate consequence, since irreducibility implies non-defectiveness for
 discrete-time switching systems, one gets the following.
\begin{cor}\label{c.10}
 If the family of matrices $\cA$ is irreducible, then 
 the weighted system $(\cA,\balpha)$ is non-defective for every weight $\balpha$. 
 \end{cor}

\begin{remark}
The proof of
Proposition~\ref{p.30} is based on the remark that 
non-defectiveness of $(\cA,\balpha)$ is independent of the weight $\balpha$ if $\rho(\cA)=1$. 
Corollary~\ref{c.10} 
identifies another class of families of matrices for which non-defectiveness is independent of the weight.  
 The property is however false for a general family $\cA$,
 as illustrated by the following example. 
 
 Consider 
 $\cA=\{A_1,A_2\}$ with
 \[A_1=\begin{pmatrix} 3&1\\ 0&3\end{pmatrix},\qquad A_2=\begin{pmatrix} 2&0\\ 0&2\end{pmatrix}.\]
Notice that the two matrices commute. For a weight 
$\balpha$, a positive integer $k$,  and $A(j) \in \cA$, $j = 1, \ldots , k$, one has that 
 \begin{equation}\label{eq:ND}
  \bigl\|A(k)\cdots A(1)
\bigr\|^{\ \frac{1}{\alpha (k)+ \cdots + \alpha(1)}}=\Big(2^m3^{k-m}
\left\| \begin{pmatrix} 1&\frac{k-m}3\\ 0&1\end{pmatrix} \right\|\Big)^{\frac1{m\alpha_2+(k-m)\alpha_1}},
\end{equation}
 for some integer $m\leq k$.
 The right-hand side of \eqref{eq:ND} can be rewritten as 
 $e^{f_{\balpha}(x)+\zeta(k)}$,
  where $x=m/k\in[0,1]$, $f_{\balpha}$ is defined by 
$$
 f_{\balpha}(x)=\frac1{\alpha_1+x(\alpha_2-\alpha_1)}
 \left(\ln(3)+\ln\left(\frac23\right)x \right),
 $$
 and $\zeta(k)\to 0$ as $k\to\infty$. 
A simple computation shows that 
the maximum of 
$f_{\balpha}(\cdot)$ is reached at $x=0$ if $\balpha=(1,1)$ and at $x=1$ if
$\balpha=(2,1)$. 
As a consequence, 
$\rho(\cA,(1,1))=3$ and  $\rho(\cA,(2,1))=2$.
Moreover, the maximum in \eqref{eq:ND}   goes as $\|A_1^k\|^{\frac1 k}$ as $k\to \infty$ if $\balpha=(1,1)$  and
as $\|A_2^k\|^{\frac1 k}$ if $\balpha=(2,1)$. 
Hence, $(\cA,(1,1))$ is defective since $3^{-k}A_1^k$ is not bounded as $k$ tends to infinity, while 
$(\cA,(2,1))$ is non-defective since $2^{-k}A_2^k$ does not depend on $k\geq 1$.
%
  \end{remark}

 The following theorem extends the main facts on extremal and invariant norms
 from classical 
 discrete-time switching systems to weighted systems. 
\begin{theorem}\label{th.20}
\textbf{a)} For a weighted system $(\cA, \balpha)$ and $\lambda>0$, we have $\rho(\cA, \balpha) < \lambda$ if  and only if there exists a norm in $\mathbb{R}^d$ such that,  in the 
corresponding operator norm, we have $\|A_j\| \, < \, \lambda^{\, \alpha_j}, \, 
j  = 1, \ldots , m$; 
\smallskip 

\textbf{b)} If a weighted system is non-defective, then 
it possesses an \textbf{extremal norm}, for which 
$$
\max_{j= 1, \ldots , m}\|\rho^{-\alpha_j}A_j\| \quad \le \quad 1, 
$$
 where $\rho = \rho(\cA, \balpha)$;  
\smallskip

\textbf{c)}  If a weighted system is irreducible, then 
it possesses an \textbf{invariant norm}, for which 
$$
\max_{j= 1, \ldots , m}\|\rho^{-\alpha_j}A_j x\| \quad = \quad  \|x\|, \qquad  x \in \re^d, 
$$
where $\rho = \rho(\cA, \balpha)$. 
\end{theorem}
\begin{proof}
\textbf{a)} If $\|A_j\| < \lambda^{\alpha_j}$ for all 
$A_j \in \cA$, then this inequality still holds after 
replacing $\lambda$ by $\lambda-\varepsilon$, whenever $\varepsilon > 0$
is small enough. Using submultiplicativity of the matrix norm, we obtain 
for all matrix products: 
$\|A(k)\cdots A(1)\| <  (\lambda-\varepsilon)^{\, \alpha(k) + \cdots + 
\alpha(1)}$. Hence, $\rho(\cA, \balpha)\le \lambda - \varepsilon < \lambda$.
Conversely, if $\rho(\cA, \balpha)< \lambda$, then 
 the family $\cA' =\delta^{\balpha}_{\lambda}(\cA)$ has joint spectral radius smaller than one. 
Hence, there is a norm in $\re^d$ such that $\|A'\|< 1$ for every $A'\in \cA'$. Therefore, $\|A_j\| < \lambda^{\alpha_j}$ for $j=1,\dots,m$. 

In the same way we prove the existence of the extremal and invariant norms in items
 \textbf{b)} and \textbf{c)} 
merely by passing to the system $\delta_{\rho^{-1}}^\alpha(\cA)$,
whose joint 
spectral radius is equal to one and by applying the classical existence results 
of  extremal and invariant norms of usual discrete-time switching systems.
%

\end{proof}
\subsection{Numerical aspects}

\subsubsection{The algorithm of Gripenberg adapted for weighted systems}
 
 Theorem~\ref{th.20} allows one to compute the weighted joint spectral radius by 
constructing the corresponding norm in~$\re^d$. We begin with 
the branch-and-bound algorithm of Gripenberg~\cite{G} for the approximate 
computation of the joint spectral radius and then consider the invariant polytope algorithm for its exact computation. 
To generalize Gripenberg's algorithm to weighted systems we 
need an extension of Item~\textbf{a)} of Theorem~\ref{th.20} to cut sets 
of matrix products (defined below). 

Consider the tree~$\bT$ of matrix products. 
The root is the identity matrix~$\mathrm{Id}$. It has $m$ children $A_j, \, j = 1, \ldots , m$. 
They form the first level of the tree. The further levels are constructed by induction. 
Every vertex (product) in the $k$th level $\Pi = A(k)\cdots A(1)$
has $m$ children $A_j\Pi$, $j = 1, \dots , m$, in the $(k+1)$-th level. 
For a vertex $\Pi = A(k)\cdots A(1)$,  we denote 
\[
|\Pi| = \alpha(k) + \cdots + \alpha(1).
\] 

A finite set of vertices on positive levels is called a {\em cut set} if    
it intersects every infinite path starting at the root (all paths are without backtracking). It is shown easily that for every cut set $\cS$, each infinite product
of matrices from $\cA$ 
is an infinite product of vertices from $\cS$. 

\begin{prop}\label{p.40}
If the tree~$\bT$ of a weighted system $(\cA, \balpha)$
possesses a cut set~$\cS$ such that 
for every vertex $\Pi \in \cS$, we have 
$\|\Pi\|\, < \, \lambda^{\, |\Pi|}\, $ for some $\lambda>0$, then $\, \rho(\cA, \balpha) < \lambda$. 
\end{prop}
\begin{proof}
The inequalities $\|\Pi\|\, < \, \lambda^{\, |\Pi|}\, $  still hold after 
replacing $\lambda$ by $\lambda-\varepsilon$, whenever $\varepsilon > 0$
is small enough. Let $n$ be the maximal level of vertices from~$\cS$. 
Then every  product $\Pi_N$ of length $N>n$ can be presented as a product 
$\Pi$ of several vertices from $\cS$ times some product  $\Pi_r$ of length $r < n$. 
Using submultiplicativity of the matrix norm, we obtain 
$\|\Pi\| \, <  \, (\lambda-\varepsilon)^{|\Pi|}$. 
Therefore, 
$$
\|\Pi_N\| \ \le \ (\lambda-\varepsilon)^{|\Pi|}\|\Pi_r\|\ = 
 \ (\lambda-\varepsilon)^{|\Pi_N|}\|\Pi_r\|\, (\lambda - \varepsilon)^{\, - |\Pi_r|}\ \le \  C\, (\lambda-\varepsilon)^{|\Pi_N|},
$$
where $C$ is the maximum of numbers $\|\Pi_r\|\, (\lambda - \varepsilon)^{\, - |\Pi_r|}$ over all products $\Pi_r$ of length $\le n$. 
Since this holds for all long  products $\Pi_N$, 
we conclude that $\rho(\cA, \balpha) \le \lambda - \varepsilon < \lambda$. 
\end{proof}


We now provide details of the algorithm.
We choose a small $\varepsilon> 0$ and define the starting value of $\lambda$ as
$\lambda = \max \{ [\rho(A_i)]^{1/ \alpha_i} \mid i = 1,\ldots , m\}$,
where $\rho(A)$ denotes the spectral radius of the matrix $A$. 
Then we go through the  tree $\bT$ starting from the first level.

For every vertex 
   $\Pi = A(k) \cdots A(1)$ on $\bT$, we compute
$\|\Pi\|^{1/|\Pi|}$ and if it  is smaller than $\lambda + \varepsilon$, then we remove from $\bT$ the vertex $\Pi$ together with the whole branch starting from it. 
This vertex is said to be dead and it does not produce children. Otherwise, we 
keep $\Pi$ 
and we go to the next vertex. 

If the value $\bigl[\rho\bigl(\Pi\bigr)\bigr]^{1/ |\Pi|}$ is bigger than $\lambda$, then we replace $\lambda$ by this value and continue. 
Otherwise, $\lambda$ stays the same. 

The algorithm terminates when there are no new alive vertices. 
Then we have $\lambda \le  \rho(\cA,\balpha) \le \lambda + \varepsilon$. 
For the last alive vertex-product $\Pi_n$, we have 
$$
\bigl[\rho\bigl(\Pi_n\bigr)\bigr]^{1/ |\Pi_n|} 
\quad  
\le \quad \rho(\cA,\balpha) \quad  \le \quad \bigl\|\Pi_n\bigr\|^{1/ |\Pi_n|}
\quad  + \quad 
\varepsilon \, . 
$$

%

\subsubsection{The invariant polytope algorithm}

The algorithm tries to find a s.m.p. ({\em spectral maximizing product}), i.e., 
a product $\Pi = A(k)\cdots A(1)$ such that 
$\rho(\cA,\balpha)\, = \, \bigl[ \rho(\Pi) \bigr]^{1/|\Pi|}$.
If this is done, then 
the weighted joint spectral radius is found. For discrete-time switching systems,
numerical experiments 
demonstrate~\cite{GP13, GP16, Mej} that for a vast majority of matrix families, 
a s.m.p. exists and the invariant polytope 
algorithm finds one.

The first step is to fix some integer 
$\ell$ and find a simple product
(i.e., a product which is not a power of a shorter product)~$\Pi = A(n)\cdots A(1)$ with the maximal value $[\rho(\Pi_n)]^{\frac{1}{|\Pi_n|}}$ among all 
products~$\Pi_n$
of lengths $n \le \ell$. We denote this value by $\rho_{c}$ and 
 call this product a {\em candidate for s.m.p.}. Next, we try 
  to prove that it is a real
s.m.p.. We normalize all the matrices $A_i$ as $\tilde A_i = \rho_c^{-\alpha_i}A_i$.
Thus we obtain the 
system~$(\tilde A, \balpha ) $ and the product $\tilde \Pi = 
\tilde A(n)\cdots \tilde A(1)$ such that $\rho(\tilde \Pi) = 1$. 
We are going to check whether  $\rho(\tilde \cA , \balpha) \le 1$. 
If this is the case, then $\rho(\tilde \cA , \balpha) = 1$. By Theorem~\ref{th.10},
we equivalently need to show that $\rho(\tilde \cA) \le 1$. 
This can be done by presenting  a polytope~$P \subset \re^d$
such that $\tilde A_i P \subset P$ for all $i = 1, \ldots , m$. 
The construction is provided in~\cite{GP13}. If the algorithm terminates within finite time, 
then it produces the desired polytope~$P$. Otherwise, we need to 
look for a different candidate for s.m.p..
\smallskip 

The criterion for terminating the algorithm  in a finite number of steps
 uses the notion of dominant product which is
a strengthening of the s.m.p.~property. A product $\Pi = A(n)\cdots A(1)$ 
is called {\em dominant} for the weighted family $(\cA,\balpha)$
if there exists a constant $\gamma < 1$ such that the spectral radius of each product of matrices
from the normalized family $\tilde \cA 
$ which is neither a power of~$\tilde \Pi$ nor that of its cyclic permutation is smaller than~$\gamma$. A dominant product is an s.m.p., but, in general, the converse is not true. 

\smallskip

\begin{theorem}\label{th.30} For a given weighted system  and for a given initial product~$\Pi$,
the invariant polytope algorithm (Algorithm~\ref{algo:IPA}) terminates within finite time if and only if~$\Pi$ is dominant and its
leading eigenvalue is unique and simple.
\end{theorem}
The proof of the theorem is similar to the proof of the corresponding theorem for
 usual discrete-time systems~\cite{GP13} and we omit it. 

The algorithm follows (here absco denotes the absolutely convex hull of a set).

\begin{algorithm}\label{algo:IPA}
\KwData{$\cA = \{ A_1,\ldots,A_m \}, \ \balpha=\{\alpha_1,\ldots,\alpha_m \}, \ k_{\max} > 0$ ($k_{\max}$ may be very large),
and a candidate spectrum maximizing product $\Pi$}
\KwResult{The invariant polytope $P$, spectrum maximizing products, the weighted joint spectral
radius~$\rho(\cA,\balpha)$}
\Begin{
\nl Set $\rho_c := \rho(\Pi)^{1/|\Pi|}$ and $\tilde{\cA} :=\delta_{\rho_c^{-1}}^\balpha(\cA)$\;
\nl Compute $v_0$, leading eigenvector of $\Pi$\;
\nl Set $k=1$\;
\nl Set $E=0$ and $\cV_{0} = \{ v_0 \}$\;
\nl \While{$E=0$ \ {\rm and} \ $k \le k_{\max}$}{
\nl Set $\cV_{k} = \cV_{k-1}, \, \cR_{k} = \emptyset$\;
\nl \For{$v \in \cR_{k-1}, \, \mbox{{\rm \textbf{and for}}} \ i = 1, \ldots , m$}{
\nl \eIf{$\tilde A_iv \in {\rm int} ({\rm absco}(\cV_{k}))$}{
\nl Leave $\cV_k, \cR_k$ as they are\;}{
\nl Set $\cV_{k} := \cV_{k}\cup \{\tilde A_i v\}, \, \cR_{k} := \cR_{k}\cup \{\tilde A_i v\}$\;}}
\nl \eIf{$\cR_{k} = \emptyset$}{
\nl Set $E=1$ (the algorithm halts) \; }{
\nl Set $k:= k+1$ \;}
}
\nl \eIf{$E=1$}{
\nl \Return{$P := {\rm absco}(\cV_k) \ $ is an invariant polytope;\\
$\hspace{14mm}  \Pi$ is a s.m.p.;\\
$\hspace{14mm} \rho(\cA,\balpha) = \rho_c$  is the weighted joint spectral radius\;}}
{{\bf print} {\rm Maximum number of iterations reached}\;}
}
\caption{The invariant  polytope algorithm \label{algoP}}
\end{algorithm}

Variants for Algorithm \ref{algoP} can be considered in the case where
there are several spectrum maximizing products.

\begin{ex} \rm
Consider the weighted system $(\cA, \balpha)$ with $\cA = \{ A_1, A_2 \}$ and
$\balpha = \{ \alpha_1, \alpha_2 \}$,
\begin{eqnarray*}
A_1 = \left( \begin{array}{rr} 1 & 1 \\ 0 & 1 \end{array} \right), \qquad
A_2 = \frac45 \left( \begin{array}{rr} 1 & 0 \\ 1 & 1 \end{array} \right).
\end{eqnarray*}
In the usual case when $\alpha_1=\alpha_2=1$ it is well-known that
\[
\rho_0 = \rho(\cA, \{ 1,1 \}) = \rho\left( A_1 A_2 \right)^\frac12 = 
1 + \frac{\sqrt5}{5} = 1.44721\ldots
\]
which implies that $A_1 A_2$ is a 
spectral maximizing product.

However, setting $\alpha_1=1$ and $\alpha_2=2$ we compute the s.m.p. $\Pi = A_1^2 A_2$, that gives $|\Pi|=4$,
\[
\rho_c = \rho(\cA, \{ 1,2 \}) = \rho(\Pi)^\frac14 =  1.314496347291999 := \frac{1}{\lambda}, \lambda = 0.760747644571326. 
\]
This gives the normalized product
\[
\tilde\Pi = \left( \lambda^{\alpha_1} A_1 \right)^2 \lambda^{\alpha_2} A_2 = {\tilde A}_1^2\,{\tilde A}_2
\]
such that $\rho(\tilde\Pi) = 1$, where ${\tilde A}_1 = \lambda^\alpha_1 A_1$ and ${\tilde A}_2 = \lambda^\alpha_2 A_2$.
As expected $\rho_c < \rho_0$.

An extremal norm is computed by the Invariant polytope algorithm 
and corresponds to a polytope with
vertices $\{ \pm v_0,  \pm v_1,  \pm v_2,  \pm v_3,  \pm v_4, \pm  v_5, \pm  v_6 \}$ where
\[
v_0 = \left( \begin{array}{l} 1 \\ 0.366025403784439 \end{array} \right) 
\]
is the leading eigenvector of $\tilde\Pi = {\tilde A}_1^2 {\tilde A}_2$ and
\begin{eqnarray*}
v_1 = {\tilde A}_1\,v_0, \quad v_2 = {\tilde A}_2\,v_0, \quad 
v_3 = {\tilde A}_1\,v_1, \quad v_4 = {\tilde A}_1\,v_2, \quad v_5 = {\tilde A}_1\,v_3, \quad v_6 = {\tilde A}_2\,v_4.
\end{eqnarray*}
\begin{figure}
\begin{center}
\includegraphics[width=0.5\textwidth]{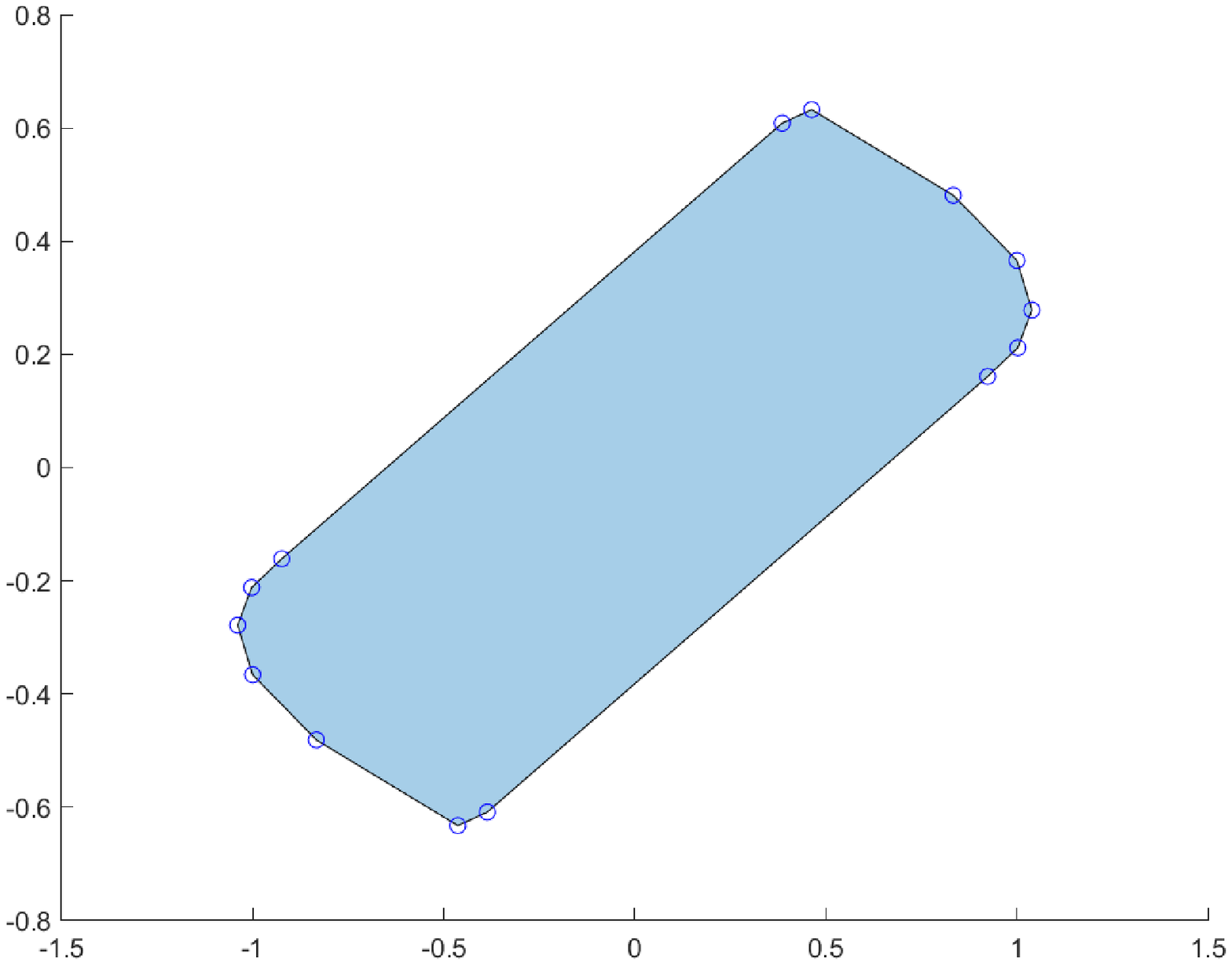}
\caption{The figure shows the extremal polytope norm computed for Example 1.}
\end{center}
\end{figure}
\end{ex}

\section{Mixed (discrete-continuous) systems}\label{sec:3}

\subsection{Theoretical aspects}

If all dwell times of a continuous-time switching system are fixed, then 
the latter is equivalent to a weighted system.
But what if the dwell times of only a part of modes are fixed while the 
other dwell times have no restrictions? In this case our equivalent system 
includes both continuous and discrete part. 
This motives the following construction. 
\smallskip 

Let $\cA = \{A_1, \ldots , A_m\}$ be a finite family of $d\times d$ matrices equipped with 
positive weights $\balpha = \{\alpha_1, \ldots , \alpha_m\}$ 
and let 
$\cB$ be a 
bounded
set of $d\times d$ matrices. For every sequence of matrices 
$\{A(i)\}_{i \in \n}$, from the family $\cA$, $\alpha(i)$ denotes the weight 
of $A(i)$.  
 
\smallskip

\begin{defi}\label{d.100}
The {\em mixed system} associated with the 
 triple 
$(\cA, \balpha , \cB)$ is the linear switching system 
having the following set of trajectories. 
Consider any  sequence 
$\bigl\{\, \bigl(\, (a_i, b_i)\, , A(i)\, 
\bigr)\, \bigr\}_{i\in D}$, 
where each $A(i)$ is a matrix from $\cA$ and $(a_i, b_i) \subset \re_+$ an open interval of length $\alpha(i)$. This sequence may be 
infinite ($D = \n$), finite ($D = \{1, \ldots , n\}$), or empty ($D = \emptyset$). The sequence of intervals increases, i.e., $b_i \le a_{i+1}$ for each $i$. 
The union of those intervals is called a {\em dark domain}, its complement in 
$\re_+$ is called an {\em active domain} and denoted by $\cT$. 
For any measurable function $B: \, \cT \to \cB$, we consider 
the system of differential and difference equations 
\begin{equation}\label{eq.mix}
\left\{
\begin{array}{cclcl}
\dot x(t) &  = & B(t)x(t)\,, && t \in \cT, \\[1mm]
x(b_i)  &  = & A(i)x(a_i)\,, & & i \in D.
\end{array}
\right. 
\end{equation}
Every solution $x: \cT \to \re^d$  of this system is called a {\em trajectory of the mixed system $(\cA, \balpha , \cB)$}, 
whose 
associated {\em switching law} is given by 
the sequence $\bigl\{\, \bigl(\, (a_i, b_i)\, , A(i)\, 
\bigr)\, \bigr\}_{i\in D}$ together with the 
function $B: \cT \to \cB$. We use ${\it sw}$
to denote any such switching law and ${\cal{S}\cal{W}}$ for the set of all switching laws associated with $(\cA, \balpha , \cB)$.
\end{defi}
\smallskip 

Clearly, the classical continuous and discrete-time switching systems are 
special cases of mixed systems.

Every trajectory of a mixed system 
is uniquely 
determined by the switching law ${\it sw}$ and by the initial point $x(0)$.
The trajectory 
has its own active domain $\cT=\cT({\it sw})$, where it is defined. 
Thus, for a mixed system, a trajectory is not a function from $\re_+$ to $\re^d$, 
but a function from a certain closed subset $\cT \subset \re_+$
to $\re^d$.  If $t \notin \cT$, then $x(t)$ is not defined. 
The dark domain $\re_+\setminus \cT$ consists of the union of the intervals 
$(a_i, b_i)$.
The transfer of the trajectory from the state $x(a_i)$ to $x(b_i) = A(i)x(a_i)$
is called a {\em jump}, and $a_i$ is a {\em jump point}. 
The set of all trajectories will be denoted by $\cX$. 

\begin{remark}\label{rem:conca}
Let $(\cA, \balpha , \cB)$ be a mixed system. Then its set of switching laws ${\cal{S}}{\cal{W}}$ is shift-invariant and closed by concatenation on their active domains, i.e., given two switching laws
${\it sw}^1$ and ${\it sw}^2$ in ${\cal{S}}{\cal{W}}$
and a time $T\in \cT({\it sw}^1)$, one can concatenate 
the restriction to $[0,T)$ of ${\it sw}^1$ with 
${\it sw}^2$, 
 in such a way to provide a switching law ${\it sw}={\it sw}^1_{|[0,T)}\ast {\it sw}^2$ in ${\cal{S}}{\cal{W}}$. 
\end{remark}

\begin{ex}\label{ex.100}{\em An important special case of a mixed system is a continuous-time 
linear switching system $\dot x   = B(t)x$,  $t \in \re_+$, 
where for each $t$, the matrix $B(t)$ is from a finite  set of matrices 
$\cB = \{B_1, \ldots , B_n\}$, 
and for several of them, say, for   $B_1, \ldots , B_m$, 
the dwell times are fixed. 
This means that each matrix $B_i$, $1\leq i \leq m,$ can 
be activated only for a time interval of a prescribed length $\alpha_i > 0$ (or positive multiples of it).
By setting $A_i = e^{\alpha_i B_i}$, $i \le m$, we obtain 
a mixed system $(\cA,\balpha,\cB)$ with 
$\cA = \{A_1, \ldots , A_m\}$, $\balpha = \{\alpha_1, 
\ldots , \alpha_m\}$, and $\cB = \{B_{m+1}, \ldots , B_n\}$. However, 
not every mixed system has this form, because not all matrices can be presented as matrix exponentials. 
}
\end{ex}

The definitions of stability and of Lyapunov exponent are 
directly extended to mixed systems.   
\begin{defi}\label{d.110}
  A mixed system $(\cA, \balpha , \cB)$ is stable if every   
  trajectory is bounded on its active domain. It is asymptotically stable 
  if for every trajectory $x$, we have $x(t) \to 0$ as ${t \to \infty, \ 
  t \in \cT}$. 
\end{defi}
The previous definition makes sense since clearly $\cT$ contains an increasing sequence of points tending to infinity. 
\begin{defi}\label{d.120}
  The Lyapunov exponent of a mixed system is the quantity 
\[
\sigma (\cA, \balpha , \cB) \quad = \quad 
\inf\,  \Bigl\{ \beta \in \re \ \Bigr| \ 
\limsup_{t\to\infty,\ t\in\cT}\frac{\ln(\|x(t)\|)}t\, \le \,\beta
\, , \ x \in \cX\,   \Bigr\}.
\]
\end{defi} 
The properties of the Lyapunov exponent $\sigma$ are similar to those 
for the classical discrete-time or continuous-time linear switching systems. 
For example, one can easily establish the following ``shift identity": 
for every $\tau \in \re$, we have 
\begin{equation}\label{eq.shift}
\sigma\, \Bigl( \delta_\tau^\balpha(\cA) \, , \balpha\, 
, \, \cB + \tau I \, \Bigr)\quad = \quad \tau \ + \  \sigma \Bigl(\cA, \balpha, \cB \Bigr)\, . 
\end{equation}
\smallskip

In the following lemma, we prove a Fenichel type of result for mixed systems, namely that asymptotic stability and exponential stability are equivalent properties for a mixed system.
\begin{lemma}\label{le:fenichel}
Let  $(\cA, \balpha , \cB)$ be a mixed system and define
\[\hat \sigma(\cA, \balpha , \cB)=\inf\,  \Bigl\{ \beta \in \re \ \Bigr| \ 
\exists C>0\mbox{ s.t. }\|x(t)\|\le C e^{\beta t}\|x(0)\|\mbox{ for all }x \in \cX\,,\ t\in\cT  \Bigr\}.
\]
Then 
\begin{description}
\item[{\bf a)}] 
$ \sigma(\cA, \balpha , \cB)=\hat \sigma(\cA, \balpha , \cB)$;
\item[{\bf b)}]
$(\cA, \balpha , \cB)$ is asymptotically stable if and only if $\sigma(\cA, \balpha , \cB)<0$.
\end{description}
\end{lemma}
Before providing a proof, let us introduce the next definition.
\begin{defi}[Interpolation of a trajectory of a mixed system]\label{d.121}
 Let $x:\cT\to \mathbb{R}^d$ be a trajectory of a mixed system $(\cA, \balpha , \cB)$ with active domain $\cT$ and dark domain $
 \re_+\setminus\cT$. The interpolation $\hat x$ of $x$ is the curve $\hat x:[0,\infty)\to \mathbb{R}^d$ defined as $\hat x=x$ on $\cT$ and by linear interpolation on the dark domain
 , i.e., $\hat x(t)=x(a)+\frac{t-a}{b-a}(x(b)-x(a))$ for $t$ in a connected component $(a,b)$ of 
 $ \re_+\setminus\cT$.
 We use $\widehat{\cX}$ to denote the set of all interpolated trajectories. 
\end{defi}

It is clear that any interpolation $\hat x$ of a trajectory $x$ of a mixed system is continuous and piecewise $C^1$ with a derivative verifying the following property: there exists a positive constant $C$ only depending on $(\cA, \balpha , \cB)$ such that 
$\Vert\dot {\hat x}(t)\Vert\leq C\Vert \hat x(t)\Vert$ for a.e. $t\in\cT$ and $\Vert\dot{\hat {x}}(t)\Vert\leq C\Vert \hat x(a)\Vert$ if $t\in(a,b)$ for every connected component $(a,b)$ of $ \re_+\setminus\cT$.
%
Based on such a property we deduce the following compactness result.

\begin{lemma}\label{lem:AA}
Let $\cB$ be compact and convex and consider $K\subset \re^d$ compact. 
Then for every sequence $({\it sw}_k)_{k\in\mathbb{N}}\subset {\cal S}{\cal W}$  
and every sequence $(x_k)_{k\in\mathbb{N}}\subset K$, there exist 
${\it sw}\in {\cal S}{\cal W}$  
and $x\in K$ such that, up to subsequence, $\hat x(\cdot;x_k,{\it sw}_k)\to \hat x(\cdot;x,{\it sw})$ uniformly on 
$[0,T]$ for every $T>0$. Moreover, for $T>0$, 
$\cT({\it sw}_k)\cap [0,T+1/k] \to \cT({\it sw})\cap [0,T]$ in the sense of the Hausdorff distance. 
\end{lemma}
\begin{proof}
Let us start by noticing that, given $T>0$, up to subsequence, 
$\cT({\it sw}_k)\cap [0,T+1/k]$ converges to the complement in $[0,T]$ of a finite number of intervals of the type $(a,a+\alpha)\cap [0,T]$ with $\alpha\in\balpha$, 
since the  number of connected components 
of the dark domain $\re_+\setminus \cT({\it sw}_k)$ intersecting $[0,T]$ is uniformly bounded. 
By a diagonal argument, the convergence holds for every $T>0$. 

Let us now deduce the first part of the statement from Arzel\`a--Ascoli theorem, by checking that  
the restrictions to $[0,T]$ of trajectories from $\widehat{\cX}$ starting in $K$ 
form a closed, uniformly bounded and equicontinuous set. 
Uniform boundedness is clear from the finiteness of $\cA$ and the boundedness of $\cB$, while
equicontinuity follows from the remark before the lemma.
Finally, closedness is a consequence  of 
the well-know corresponding property in the case $\cA=\emptyset$ and 
the convergence up to subsequence of the active domains. 
\end{proof}

\begin{proof}[{Proof of Lemma~\ref{le:fenichel}}] 
%
It is clear that 
$ \sigma(\cA, \balpha , \cB)\le \hat \sigma(\cA, \balpha , \cB)$
and that $\sigma(\cA, \balpha , \cB)<0$ implies asymptotic stability. 

The lemma is proved if we show that asymptotic stability implies that 
$\hat \sigma(\cA, \balpha , \cB)<0$. 
Indeed, by means of \eqref{eq.shift} and together with the trivial implication in Item~{\bf a)}, this shows that if 
$ \sigma(\cA, \balpha , \cB)<\lambda$ for some $\lambda\in \mathbb{R}$, then 
 $\hat \sigma(\cA, \balpha , \cB)<\lambda$ is also true, that is, $\hat \sigma(\cA, \balpha , \cB)\le \sigma(\cA, \balpha , \cB)$. 

%
Let $S$ be the unit sphere of $\mathbb{R}^d$ for the norm $\|\cdot\|$.
We claim that there exists a time $T>0$ such that, for every $x\in S$ and ${\it sw}\in {\cal S}{\cal W}$, one has $\Vert {x}(t;x,{\it sw})\Vert\leq 1/2$ for some $t\in [0,T]\cap \cT({\it sw})$. Indeed, arguing by contradiction, one should have that for every $k\in \mathbb{N}$ there exist
$x_k\in S$ and ${\it sw}_k\in {\cal S}{\cal W}$
such that 
\begin{equation}\label{unmezzo}
\|{x}(t;x_k,{\it sw}_k)\|\ge1/2
\end{equation}
 for every $t\in[0,k]\cap \cT({\it sw}_k)$. 
Denoting by $\overline{{\rm co}}(\cB)$ the closure of the convex hull of $\cB$,
by Lemma~\ref{lem:AA} there exist
a trajectory ${x}_{*}$ of $(\cA,\balpha,\overline{{\rm co}}(\cB))$ with active domain $\cT$  
such that
$\|{x}_*(t)\|\ge1/2$ for every $t\in \cT$. 
Hence $(\cA,\balpha,\overline{{\rm co}}(\cB))$ is not asymptotically stable, which, by a standard approximation argument (see \cite{IngallsSontagWang}), contradicts  the asymptotic stability of 
$(\cA,\balpha,\cB)$
and, thus, proves the claim.
%

One easily deduces from the claim and the 
shift-invariance property observed in Remark~\ref{rem:conca} that 
there exists $C>0$ such that 
$\|{x}(t)\|\leq C 2^{-t/T}\|{x}(0)\|$ for every ${x}\in\widehat{\cX}$ and every $t\in\cT$, 
concluding the proof of the lemma.  
\end{proof}

The notions of non-defectiveness and irreducibility 
extend to mixed systems as follows.
\begin{defi}\label{d.31}
A mixed system $(\cA, \balpha , \cB)$ is said to be 
\begin{itemize}
\item 
\emph{non-defective} if there exists a positive constant $C$ such that
$\|x(t)\|\leq Ce^{\sigma t}\|x(0)\|$  for every $x\in \cX$ and every $t\in \cT$,
where $\sigma = \sigma(\cA, \balpha,\cB)$;
\item 
\emph{irreducible} if $\cA\cup \cB$ is irreducible.
\end{itemize}
\end{defi}

Let us note that a trajectory of a mixed system may reach the origin at some 
time~$\tilde t<\infty$, after which it stays at the origin forever. This situation is impossible for continuous-time systems, but for mixed systems it can happen, provided that one of the matrices $A_j$ is degenerate. 

Given a trajectory $x(\cdot)$ of the mixed system and $f:\re^d\to \re_+$ positive define,
we 
say that \emph{$f(x(t))$   strictly decreases in $t$} if 
$f(x(t_1)) > f(x(t_2))$ for every $t_1, t_2 \in \cT$ such that 
$t_1 < t_2$ and $x(t_1) \ne 0$. 
Thus, the value $f(x(t))$ decreases 
not on the whole $\re_+$, where it may not be defined, but on the active domain. 
Moreover, if the trajectory stabilizes at zero at some time $\tilde t$, then we require $f(x(t))$ to strictly decrease only for $t < \tilde t$.

We now formulate the main theorem on extremal and invariant norms for mixed systems. 

\begin{theorem}\label{th.40}
Let $(\cA, \balpha , \cB)$ be a mixed system and set $\sigma= \sigma(\cA, \balpha,\cB)$. Then the following holds:

\textbf{a)} For $\lambda\in\mathbb{R}$, $\sigma
< \lambda$ if  and only if there exists a norm in $\re^d$ such that
for every trajectory $x\in\cX$, the function $\|e^{-\lambda t} x(t)\|$
strictly decreases on $\cT$.   In the 
corresponding operator norm, we have $\|A_j\| \, < \, e^{\alpha_j \lambda } \, , \, 
A_j \in \cA$, and for each $x$ in the unit sphere of this norm, all vectors $(B - \lambda I) x, \, B\in \cB, $ starting at $x$, are directed 
inside the unit sphere (i.e., $\|x+\varepsilon  (B - \lambda I) x\|<1$ for every $\varepsilon>0$ small enough).

\textbf{b)} If $(\cA, \balpha , \cB)$ is non-defective, then 
it has an \textbf{extremal norm}, for which 
every trajectory  possesses the property 
$\|x(t)\| \le e^{\sigma t}\|x(0)\|, \, t \in \cT$.  

\textbf{c)}  If $(\cA, \balpha , \cB)$ is irreducible and $\cB$ is compact and convex, then 
it possesses an \textbf{invariant norm}, for which 
all trajectories 
satisfy 
$\|x(t)\| \le e^{\sigma t}\|x(0)\|, \, t \in \cT$, and for every $x_0 \in \re^d$ there exists a 
trajectory $\bar x$ starting at $x_0$ such that 
$\|\bar x(t)\| = e^{\sigma t}\|x_0\|$,  $t \in \cT$. 
\end{theorem}

Item~\textbf{a)} of Theorem~\ref{th.40}, together with Item~\textbf{b)} of Lemma~\ref{le:fenichel}, immediately implies the following. 
\begin{cor}\label{c.20}
A mixed system is asymptotically stable if and only if there exists 
a norm $\|\cdot \|$ in $\re^d$ such that for every trajectory $x\in\cX$, 
$\|x(t)\|$ strictly decreases in $t$. 
 \end{cor}

On the other hand,  Item~\textbf{c)} of Theorem~\ref{th.40} has the following geometrical interpretation.

\begin{cor}\label{c.30}
Let $(\cA, \balpha , \cB)$ be an irreducible mixed system with $\sigma(\cA, \balpha , \cB)=0$ and $G$ be the 
unit ball of the invariant norm given in Item~\textbf{c)} of Theorem~\ref{th.40}. Then every trajectory 
starting in 
$G$ never leaves $G$. On  the other hand, if $\cB$ is compact and convex, then 
for every point $x_0$ in the boundary of $G$, there exists 
a 
trajectory that starts at $x_0$ and lies  entirely on that boundary. 
 \end{cor}
 We next provide a proof of Theorem~\ref{th.40}. 

\begin{proof}[{Proof of Theorem~\ref{th.40}}]
We split the proof into four steps.
First we construct a special positively-homogeneous monotone convex functional~$\varphi$ (Step~1)
and prove that it is actually a norm in $\re^d$ when it is finite (Step 2). As a consequence, we 
deduce 
%
Items~\textbf{a)} and \textbf{b)}.
In Step~3 we show 
that irreducibility implies non-defectiveness, and so 
$\varphi$ is an extremal norm for irreducible systems. 
Finally, in Step 4, based on $\varphi$ we construct an invariant  norm $w$. 
In view of~(\ref{eq.shift}) it suffices to consider the case $ \lambda  = 0$
in item \textbf{a)} and $\sigma = 0$ in items \textbf{b)} and \textbf{c)}.
We can also, without loss of generality, assume that $\cB\not=\emptyset$, since otherwise $(\cA,\balpha,\cB)$ is a weighted system, for which Theorem~\ref{th.20} applies.

\smallskip 

{\tt Step 1}.  For arbitrary $t \ge 0$ and $z \in \re^d$, denote
$$
\ell(z, t) \ = \ \sup\, \Bigl\{\, \|x(t)\|\, \mid \quad  
x
\in \cX\,,
 \, x(0) = z\, , \, t \in \cT \Bigr\}\ .
$$
The supremum is taken over those trajectories whose active domain contains $t$. 
The set of such trajectories is nonempty, since we are assuming that $\cB\not=\emptyset$.

For every fixed $t$, the function $\ell(\cdot , t)$ is a seminorm on~$\re^d$, i.e., it is positively homogeneous
and convex, as a supremum of homogeneous convex functions.  The
function 
\[\varphi(z) = \sup\limits_{t \in \re_+} \ell(z, t)\]
 is, therefore,
also a seminorm  as a supremum of seminorms.
Moreover, $\varphi(z) \ge \ell(z,0) = \|z\|$, hence $\varphi(z)$ is strictly positive, whenever $z\ne 0$.
For every trajectory
$x(t)$ the function $\varphi\bigl(x(t)\bigr)$ is non-increasing in~$t$
on the set $\cT$, by the concatenation property presented in Remark~\ref{rem:conca}.
%
Thus, if $\varphi(z) < +\infty$ for all $z$, then $\varphi$ is a norm which is 
non-decreasing along every trajectory of the system. 

\smallskip 

{\tt Step 2}. If $\sigma < 0$, then $\sigma < -\varepsilon$ for some positive 
$\varepsilon$. Consider the shifted system $(\delta^\balpha_\varepsilon (\cA),\balpha,\cB + \varepsilon I)$ 
and denote by $\varphi_{\varepsilon}$ the corresponding 
function $\varphi$  for this system. 

Thanks to \eqref{eq.shift} and to Item~\textbf{a)} of Lemma~\ref{le:fenichel}, 
all trajectories 
of $(\delta^\balpha_\varepsilon (\cA),\balpha,\cB + \varepsilon I)$
are uniformly bounded, and hence $\varphi_{\varepsilon}(z) < +\infty$
for each $z \in \re^d$. Therefore, 
$\varphi_{\varepsilon}$ is a norm, which is non-decreasing 
along any trajectory of the shifted system. On the other hand, 
every trajectory of the shifted system has the form $e^{-\varepsilon t}x(t)$, 
where $x\in \cX$. 
For every $t_1, t_2 \in \cT$ such that $t_1 < t_2$, 
we have $\|e^{\varepsilon t_1}x(t_1)\|
\ge \|e^{\varepsilon t_2}x(t_2)\|$. Thus, $\|x(t_2)\| \le 
e^{\varepsilon (t_1 - t_2)}\|x(t_1)\|$. Hence, the norm $\varphi_{\varepsilon}$
strictly decreases along every trajectory $x \in \cX$. 

Now consider a new norm $\|\cdot\| = 
\varphi_{\varepsilon}(\cdot)$. For arbitrary $x_0\ne 0$ and $A_j \in \cA$, take a 
switching law with $a_1 = 0, A(1) = A_j$,  and 
take   an arbitrary trajectory starting at $x_0$.  
We have $\|A_jx_0\| = \|x(b_1)\| < \|x(a_1)\| = \|x_0\|$.
Thus, $\|A_jx_0\| < \|x_0\|$. 
Since this is true for all $x_0\ne 0$, we see that $\|A_j\| < 1$. 
This proves the first property from~\textbf{a)}. 
On the other hand, as shown in~\cite{MP1, MP2} 
each norm that decreases along any trajectory possesses the 
second property from~\textbf{a)}: for every $x$ such that $\|x\|=1$, 
all the vectors $Bx, \, B \in \cB$, starting at $x$  are directed inside the unit sphere.   This completes the proof of~\textbf{a)}.  
\smallskip 

To prove~\textbf{b)} it suffices to observe that if the system is non-defective and $\sigma=0$, 
then  $\varphi(x) < + \infty$ for all~$x$. Hence, $\varphi$ is a desired extremal norm, which 
in non-decreasing along any trajectory $x \in \cX$. 
This concludes the proof 
of~\textbf{b)}. 
\smallskip

{\tt Step 3}. Let us now tackle Item~\textbf{c)}.
We begin by proving that if the system is 
irreducible, then  $\varphi(x) < + \infty$ for all~$x$, and so $\varphi$ is a norm.
Denote by~$\cL$ the set of points $x \in \re^d$ such that $\varphi (x) < +\infty$.
Since $\varphi$ is convex and homogeneous, it follows that 
$\cL$ is a linear subspace of $\re^d$.
Let us show that $\cL$ is an invariant subspace for all operators 
from $\cA$ and from $\cB$. 
For every $z \in \cL$, each trajectory 
starting at $z$ is bounded, hence each trajectory 
starting at $Az$, $A \in \cA$, is bounded as well, as a part of the 
trajectory starting at $z$. Hence, $Az \in \cL$ and so 
$\cL$ is a common invariant subspace for the family~$\cA$.
Similarly, for every $z\in \cL$, $B \in \cB$, and $t\ge 0$, $e^{t B}(z)$ is in $\cL$, from which we deduce that the tangent vector $Bz$ is also in $\cL$. 
Thus,  $ \cL$ is a common invariant subspace for
both $\cA$ and $\cB$. 
From the irreducibility
it follows that either $\cL = \re^d$ (in which case $\varphi$ is a norm) or $\cL = \{0\}$. It remains to show that the latter is impossible.

Consider the unit sphere $S = \{x \in \re^d \ | \ \|x\|=1\}$. 
 If $\cL = \{0\}$, then $\varphi (x) = +\infty$ for all $x \in S$.
For every natural $n$ denote by $\cH_{\, n}$ the set of points $z \in S$
for which there exist a trajectory starting at~$z$ and a time 
 $T = T(z) \le n$ in the corresponding active domain $\cT$
such that $\|x(T)\| > 2$. Clearly, $\cup_{n =1}^{\infty}\cH_{\, n} \, = \, S$. 
Since each $\cH_{\, n}$ is open, the compactness of~$S$
implies the existence of a finite subcovering, i.e., 
the existence of a natural $N$ such that 
$\cup_{n = 1}^{N}\cH_{\, n} \, = \, S$. Equivalently,  $T(z) \le N$
for all $z \in S$. Thus, starting from an arbitrary point $x_0 \in S$
one can consequently build a trajectory $x\in \cX$
and an increasing sequence $(t_k)_{k\in \mathbb{N}}$ in  $\cT$
such that $\|x({t_{k+1}})\| \, > \, 2\|x({t_k})\|$ and $|t_{k+1} - t_k| \le N$
for all $k$. 
For this trajectory,  $\|x(t_n)\| \, > \,  2^n$ and $t_n \le nN$, hence $\|x(t_n)\| \, > \,
e^{t_n\ln 2 / N}$. Therefore, $ \sigma  \ge \frac{\ln 2}{N} > 0 $, which contradicts the assumption.
The contradiction argument allows to conclude that  
$\cL = \re^d$ and $\varphi$ is a norm.
\smallskip

 {\tt Step 4}. 
  We have found a norm $\varphi$ which is
 non-increasing on the active domain along every trajectory~$x\in \cX$.
 By convexity of $\varphi$, this also implies that $\varphi(\hat x(t))\le \varphi(\hat x(0))$ for every 
trajectory~$\hat x\in \widehat{\cX}$ and every $t\ge 0$. 
%
Define, for every 
$x\in \mathbb{R}^d$,
$$
w(x)=\limsup_{t\to \infty}\sup_{{\it sw}\in{\cal{S}}{\cal{W}}}
\varphi(\hat{x}(t;x,{\it sw})).
$$
The finiteness of $w$ follows from the monotonicity of $\varphi$. 
Notice that, by Remark~\ref{rem:conca}, 
\begin{equation}\label{eq:wdddd}
w(x(t;x,{\it sw}))
\leq w(x),\qquad {\it sw}\in{\cal{S}}{\cal{W}},\; t\in \cT({\it sw}).
\end{equation}
We claim that $w$ is a norm. Homogeneity is obvious and subadditivity follows form the inequality
\[
\varphi(\hat{x}(t;x+y,{\it sw}))\leq 
\varphi(\hat{x}(t;x,{\it sw}))+\varphi(\hat{x}(t;y,{\it sw})),\
\quad  {\it sw}\in{\cal{S}}{\cal{W}},\ t\ge 0.
\]
Let us assume by contradiction that $w(x)=0$ for some $x\ne0$. 
It follows from \eqref{eq:wdddd} that $w(x(t;x,{\it sw}))=0$ for all ${\it sw}\in{\cal{S}}{\cal{W}}$ and $ t\in \cT({\it sw})$. Since, moreover, 
 the linear space generated by $\{x(t;x,{\it sw})\mid {\it sw}\in{\cal{S}}{\cal{W}},\;t\in \cT({\it sw})\}$,  is invariant for $\cA\cup \cB$, then it is equal to $\mathbb{R}^d$, which implies that $w\equiv 0$ on $\mathbb{R}^d$. It follows from  Item~\textbf{b)} of Lemma~\ref{le:fenichel} that
$\sigma(\cA, \balpha,\cB)<0$, leading to a contradiction. This concludes the proof that $w$ is a norm. 

Take now $x\in\mathbb{R}^d$ and  
consider two sequences $({\it sw}_k)_{k\in \mathbb{N}}\subset {\cal S}{\cal W}$ and $(t_k)_{k\in \mathbb{N}}\subset \re_+$ such that $t_k\to \infty$ as $k\to \infty$ and 
\[w(x)=\lim_{k\to \infty}\varphi(\hat{x}(t_k;x,{\it sw}_k)).\]
Since $\varphi$ is non-increasing along trajectories, we have that 
\begin{equation}\label{eq:desp}
\liminf_{k\to\infty}\min_{t\in [0,t_k]\cap \cT({\it sw}_k)}\varphi(x(t;x,{\it sw}_k))\ge w(x).
\end{equation}

 By Lemma~\ref{lem:AA},
there exists 
 ${\it sw}\in  {\cal S}{\cal W}$
 such that, up to subsequence,
 $\hat x(\cdot;x,{\it sw}_k)$ converges to $\hat x(\cdot;x,{\it sw})$ uniformly on all compacts of $\re_+$. 
 Moreover, $\cT({\it sw}_k)$ converges to $\cT({\it sw})$ on compact intervals in the sense guaranteed by Lemma~\ref{lem:AA}.
Together with \eqref{eq:desp}, this implies that 
\[\liminf_{t\to\infty,\;t\in \cT({\it sw})}\varphi(x(t;x,{\it sw}))\ge w(x).\]
Hence, by definition of $w$, $w(x)=\lim_{t\to\infty,\;t\in \cT({\it sw})}\varphi(x(t;x,{\it sw}))$. 
We conclude the proof that $w$ is an invariant norm by deducing from \eqref{eq:wdddd} that 
 $w(x(t;x,{\it sw}))=w(x)$ for every  $t\in\cT({\it sw})$. 
 \end{proof}

Having proved the existence theorem for 
extremal and invariant norms we are now able to approximate the 
Lyapunov exponent numerically by constructing  
polytopic Lyapunov functions. 
\bigskip 

\subsection{The algorithm for mixed systems}

One of the methods to prove stability of mixed systems is by discretization.
First we assume that $\cB$ is finite.

It is well known that the joint spectral radius of a compact set $\cB$ of matrices is the same as that of its convex hull ${\rm co}(\cB)$. 
If this is finitely generated, i.e., ${\rm co}(\cB) = {\rm co} \left( \{B_1,B_2,\ldots, B_m\} \right)$ then we can apply our algorithm.
If this is not the case, one possibility would be that of finding a nearby polyhedron containing ${\rm co}(\cB)$ and apply the algorithm
to the family given by the vertexes of this polyhedral set. If the set is $\eps$-close to $\cB$ then the computed joint spectral radius
is $\eps$-close to the joint spectral radius $\rho \left( {\rm co}(\cB) \right) = \rho\left( \cB \right)$.

The idea is that of constraining (\ref{eq.mix}) by imposing that the time
instants at which switching is allowed (the switching instants) for the free
matrices (those belonging to $\cB$) are multiple of a small time-duration $\tau$. 


This procedure gives rise to a weighted system 
$(\cC_{\tau},\gamma_\tau)$ 
whose corresponding modes are the elements of $
\cA$ and those of $\cB_\tau=\{{\rm e}^{\tau B}\mid B \in \cB\}$, i.e., $\cC_\tau = \cA \cup \cB_\tau$.
The weight vector $\gamma_\tau$ is obtained
associating with any element $A_i\in \cA$ its corresponding $\alpha_i$ from $\balpha$ and 
with any matrix in $\cB_\tau$ 
the weight $\tau$.

%

We recall that 
Algorithm~\ref{algoP} 
tries to find a s.m.p. $\Pi_\tau = C(k)\cdots C(1)$, with $C(1),\dots,C(k)\in  \cC_\tau$, such that $\rho({\cC}_\tau,\gamma_\tau)=\rho(\Pi_\tau)^{\frac{1}{|\Pi_\tau|}}$.
If this is done, then 
the weighted joint spectral radius is found. 

\subsection{Lower and upper bounds for the Lyapunov exponent}

Note that, for any $\tau>0$ and for an arbitrary product $\Pi$
of matrices from ${\cC}_\tau$, the Lyapunov exponent $\sigma(\cA,\balpha,\cB)$ of system \eqref{eq.mix} 
is bounded below by the quantity
\begin{equation}\label{eq.lower1}
\beta(\Pi) \ = \ \frac{1}{|\Pi|}\log\left( \rho(\Pi) \right). 
\end{equation}
Choosing the product with the biggest $\beta(\Pi)$, we find the 
best lower bound for the Lyapunov exponent. If Algorithm~\ref{algoP}
finds the s.m.p. $\Pi_{\tau}$, then this product provides 
this best lower bound. Using the short notation $\beta(\tau) = 
\beta(\Pi_{\tau})$, we get 
\begin{equation}\label{eq.lower2}
\beta(\tau) \ \le \ \sigma(\cA,\balpha,\cB)\, . 
\end{equation}

Similarly to \cite{GLP15}, an upper bound to $\sigma(\cA,\balpha,\cB)$ is found as follows.
%
%
%
%
For an arbitrary polytope $P \subset \re^d$ symmetric about the origin
and for the 
weighted
family $(C_\tau,\gamma_\tau)$,    
we define the value $\mu (\cB , P) \, = \, \mu (P)$ as
\begin{eqnarray}
\mu (P) & = & \inf \ \Bigl\{ \mu \in \re \mid 
\mbox{for each vertex}\ v \in P \ \mbox{and} \ B \in \cB, \Bigr.
\nonumber
\\
& & \Bigl. \qquad \mbox{the vector} \ (B - \mu I)v \  \mbox{is directed inside} \,
 P\, \Bigr\}\, .
\label{def-gamma}
\end{eqnarray}
For the extremal polytope  $P_{\tau}$ computed by Algorithm~\ref{algoP}
we use the short notation~$\mu(P_{\tau}) = \mu(\tau)$. 
The following simple observation is  crucial for the further results.

\begin{prop}\label{p50}
Let $\cB$ be finite and $\tau > 0$. Then for an arbitrary symmetric polytope~$P$
and for arbitrary product $\Pi$ of matrices from the weighted
family $(C_\tau,\gamma_\tau)$, we have 
\begin{equation}\label{lu}
\beta (\Pi) \ \le \  \sigma(\cA,\balpha,\cB) \ \le \ \mu (P)\, .
\end{equation}
In particular, 
\begin{equation}\label{lu-extr}
\beta (\tau) \ \le \  \sigma(\cA,\balpha,\cB) \ \le \ \mu (\tau)\, .
\end{equation}
\end{prop}
\begin{proof}
The proof is completely analogous to the one given in \cite{GLP15}
for classical switching systems.  
\end{proof}

If, for a polytope~$P$, we have $e^{\, \tau \, B}\, P \, \subset \, \lambda^\tau \, P\, , \ B \in \cB$,
as well as $A_i P \, \subset \, \lambda^{\alpha_i} \, P\, , \ A_i \in \cA$, then 
$\lambda \ge \rho(\cC_\tau,\gamma_\tau)$. 
Clearly, if we have an extremal polytope~$P_{\tau}$ available, then we also know the value of 
the corresponding weighted joint spectral radius.
 In some cases, however, the extremality
property is a too strong requirement, and computing the invariant polytope 
may take too much time. However, to estimate the Lyapunov exponent the following weaker version of extremality suffices:
\begin{defi}\label{d17}
Given $\eps \ge 0$,
a polytope~$P$ is called {\em $\eps$-extremal} for 
$(\cC_\tau,\gamma_\tau)$ if
$$
 e^{\tau B} \, P \ \subset \ e^{\tau \eps } \, \rho (C_\tau,\gamma_\tau)^\tau\, P\, , \qquad  B \in \cB\,, 
$$
and 
$$
 A_i \, P \ \subset \ e^{\alpha_i \eps } \, \rho (C_\tau,\gamma_\tau)^{\alpha_i}\, P\, , \qquad  A_i \in \cA\, .
$$
\end{defi}

{
When $P = P_\tau$ is extremal, the double inequality~(\ref{lu-extr}) localizes the Lyapunov exponent to the
segment $[\beta(\tau) \, , \, \mu(\tau)]$. 
The length of this segment does not exceed  a linear 
function of~$\tau$. So, the precision of the estimate~(\ref{lu-extr})
is not worse than~$C\tau$. The following theorem considers a more general case, when the polytope~$P$ is not necessarily extremal but only $\eps$-extremal.
}

\begin{theorem}\label{th8}
For every finite irreducible mixed system~$(\cA,\alpha,\cB)$,  there exists a positive constant~$C$ such that
for all $ \, \eps \ge  0$ and~$\tau$ such that the family $\cC_\tau$ is irreducible,  we have
$$
\mu (P) \ - \ \beta (\tau) \ \le \ C\tau \, + \, \eps\, ,
$$
whenever~$P$ is an $\eps$-extremal polytope for~$(\cC_\tau,\gamma_\tau)$.
\end{theorem}
{
Thus, by inequality~(\ref{lu}), every discretization time~$\tau > 0$  gives the lower bound~$\beta(\Pi)$
for the Lyapunov exponent, and that discretization time with an $\eps$-extremal polytope~$P$ gives the upper 
bound~$\mu(P)$.
Theorem~\ref{th8} ensures that at least in case $\Pi = \Pi_{\tau}$, the precision of these bounds is linear in~$\tau$ and $\eps$. 
In particular, for $\eps = 0$,
we have the following.
}
\begin{cor}\label{c5}
If the polytope~$P_\tau$ is extremal for~$(\cC_\tau,\gamma_\tau)$, then $\mu (\tau) - \beta (\tau) \, \le \, C\tau$.
\end{cor}
{
\begin{remark}\label{r.lu}
{\em If one succeed in finding a ``proper'' polytope $P$ and 
a product $\Pi$ for which the difference $\mu(P) - \beta(\Pi)$ is small, 
then we have an a posteriori estimate~(\ref{lu}) for the Lyapunov exponent. 
Theorem~\ref{th8} shows that at least in the case when $P$ is $\varepsilon$-extremal and $\Pi$ is an s.m.p. the precision of this estimate decays linearly with~$\tau$. 
In most of practical cases this estimate behaves even better. 
}
\end{remark}
}
Proposition~\ref{p50} and Theorem~\ref{th8} suggest the following method of approximate computation
of the Lyapunov exponent $\sigma(\cA,\alpha,\cB)$:
\smallskip

1) choose a discretization time $\tau > 0$, and compute the weighted joint spectral radius $\rho(\cC_\tau,\gamma_\tau)$;

2) choose $\eps \ge 0$ and construct an $\eps$-extremal polytope~$P$ for~$(\cC_\tau,\gamma_\tau)$ (if $\eps=0$ we intend
an extremal polytope).
\smallskip

{
Then we localize the Lyapunov exponent $\sigma(\cA,\alpha,\cB)$ on the segment $[\beta(\tau) , \mu(P)]$, whose length
tends to zero  with a linear rate in~$\tau$ and $\eps$ as $\tau , \eps \to 0$.
}

If we consider a product $\Pi$ which is not an s.m.p. then we have to replace $\beta(\tau)$ by $\beta(\Pi)$. Let us recall
that for every product $\Pi$, $\beta(\Pi) \le \beta(\tau)$. Therefore if $\Pi$ is not an s.m.p. then the lower bound can be
worse than the optimal one considered in Theorem \ref{th8}.
In such case linear convergence as $\tau \rightarrow 0$ is not guaranteed. In practice nevertheless the difference
$\mu(P) - \beta(\Pi)$ can always be made sufficiently small to provide a satisfactory approximation of the Lyapunov exponent.

{\em Deriving the lower and upper bounds for the Lyapunov exponent}.
Thus, the first part of the algorithm produces an $\eps$-extremal polytope~$P$.
We compute $\mu(P)$ by definition, as the infimum of those numbers~$\mu$ such that
the vector $(B - \mu I) v $ is directed inside $P$, for each vertex $v \in P$ 
and for every $B \in \cB$.
This is done by taking a small~$\delta> 0$ and solving the following LP problem:
$$
\left\{
\begin{array}{l}
\mu \ \to \ \inf \\
v + \delta (B - \mu I)v \, \in \, P,\\
\, v \quad {\rm vertex \ of} \quad P, \ B \in \cB\, .
\end{array}
\right.
$$
Thus, as a result of Algorithms~\ref{algoP} and \ref{algoP2}, we obtain the lower bound $\beta(\tau)$ and
the upper bound~$\mu(P)$ for the Lyapunov exponent.
The polytope~$P$ identifies the Lyapunov norm for the family.
If $\mu (P) < 0$, then we conclude that the system is asymptotically stable 
and its joint Lyapunov function has the polytope~$P$ as unit ball.

\smallskip
\begin{algorithm}
\DontPrintSemicolon
\KwData{$\cB=\{B_1,\dots,B_M\}$, $P, V$ (system of 
vertices of $P=\mathrm{absco}(V)$, $\delta$ (small positive stepsize)}
\KwResult{$\mu(P)$}
\Begin{
\For{$i=1,\ldots,M$}{
\nl Solve the LP problems (w.r.t. $\{t_v,s_v\}$, $\mu_i$)
\begin{eqnarray}
\label{LP.r}
\begin{array}{rl}
\min & \mu_i
\\[0.1cm]
{\rm s.t.}
     & v + \delta (B_i - \mu_i I) v = \sum\limits_{w \in \cV}\,
		 t_w\,w - s_w\, w \quad \forall v \in V
\\[0.1cm]
{\rm and}
     & \sum\limits_{w \in V}\, t_w + s_w \le 1, \qquad t_w, s_w \ge 0 \quad
		   \forall w \in V
\end{array}
\nonumber
\end{eqnarray}
}
\nl Return $\mu(P) := \max\limits_{1 \le i \le m} \mu_i$ \;
}
\caption{Algorithm for computing  the best upper bound~$\mu(P)$ \label{algoP2}}
\end{algorithm}


\begin{ex} \label{example2} \rm
Let ${\cA}=\{ A_1 \}$  with $\alpha_1=1$ and ${\cB}=\{ B_1, B_2 \}$ with
\begin{eqnarray*}
A_1 = \left(\begin{array}{rr}
 0 & -\frac75
\\\frac75 & 0
\end{array} \right), \quad
B_1 = \log \left( \left(\begin{array}{rr}
 1 & 1
\\
-1 & 1
\end{array} \right) \right), \quad
B_2 = \log \left( \left(\begin{array}{rr}
 1 & 1
\\
-1 & 0
\end{array} \right) \right).
\end{eqnarray*}
For $\tau = 1$ we set exactly $\cC_{\tau} =$ $\{ A_1, e^{B_1}, e^{B_2} \}$ and $\gamma_\tau=(1,1,1)$.

By means of Algorithm \ref{algoP} we are able to prove that the product of degree equal to $5$,
$$\Pi = e^{B_2} A_1 e^{B_1} e^{B_2} A_1$$  
is spectrum maximizing and apply Algorithm \ref{algoP}
with $C_1 = A_1/\rho, C_2=e^{B_1}/\rho, C_3 = e^{B_2}/\rho$, where $\rho=\rho(\cC_\tau,\gamma_\tau)$.
As a result we obtain the polytope norm in Figure \ref{fig2} whose unit ball $P$ is a 
polytope with $16$ vertices.

Applying Algorithm \ref{algoP2} we obtain the optimal shift $\mu(P) \approx 0.65$, so that we have the estimate
$$
\beta(\tau) = 0.38\ldots \le \sigma(\cA,\balpha,\cB) \le 1.03\ldots = \mu(P_\tau).
$$
Figure \ref{fig2} illustrates the fact that the computed polytope $P$ is positively invariant for the shifted 
family $\cB - \mu(P) I$.

In order to increase the accuracy of the computation one has to reduce $\tau$.

\begin{figure}
\begin{center}
\includegraphics[angle=90,width=0.5\textwidth]{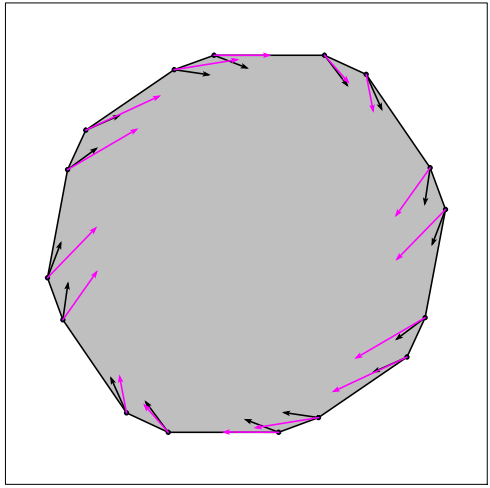}
\caption{The figure shows the extremal polytope norm computed for Example~\ref{example2}. \label{fig2}}
\end{center}
\end{figure}

\end{ex}

%

\section{Mixed  systems on graphs}\label{sec:4}

Recently many authors introduced and analysed {\em constrained discrete-time switching systems}, where not all switching laws are possible but only those satisfying certain  stationary constraints~\cite{D, K, PhJ1,  PhJ3, SFS, WRDV}. The concept slightly varies in different papers. One of the most general forms was 
considered  in~\cite{CGP}. We describe the main construction, adapted to weighted systems 
(in~\cite{CGP} this was done for the usual discrete systems, i.e., with unit weights). Then we extend it to mixed systems.
 
\bigskip 

\subsection{Discrete weighted systems on graphs}\label{sec:41}

Consider a directed strongly connected
multigraph $G$
with $n$ vertices $g_1, \ldots , g_n$. Sometimes, the vertices will be denoted by their indices.
With each  vertex~$i$ we associate a linear space $L_i$ of dimension $d_i < \infty$.
If the converse is not stated, we assume $d_i \ge 1$.  The set of spaces
$L_1, \ldots , L_n$ is denoted by $\cL$. For each vertices $i, j \in G$
(possibly coinciding),
there is a set $\cA_{ji}$  (possibly empty) of edges
from $i$ to $j$. Each edge from $\cA_{ji}$ is identified with a linear operator $A_{ji}: \, L_i \to L_j$ that has its weight $\alpha_{ji} > 0$.
Thus, we have a family of spaces~$\cL$
and a family of operators-edges  $\cA = \cup_{i, j}\cA_{ji}$ 
 that act between these spaces according to the
multigraph $G$ and have weights  $\balpha = \{\alpha_{ji}\}_{i, j}$. 
We obtain a system $\, \xi = (G, \cL, \cA, \balpha)$  made of the multigraph, the spaces, the operators, and their weights.
A path $\omega$ on the multigraph $G$ is a sequence of connected subsequent edges, its total weight (the sum of weights  of edges)
is denoted by $|\omega|$. 
With every path $\omega$ along
vertices $i_1\rightarrow i_2 \rightarrow \cdots \rightarrow i_{k+1}$ that consists of edges (operators)
$A_{i_{s+1}i_s} \in \cA_{i_{s+1}i_s}, \, s = 1, \ldots , k$, we associate the corresponding product (composition) of operators $\Pi_{\omega} = A_{i_{k+1}i_k}\cdots A_{i_2i_1}$.
Note that $|\omega| = \alpha_{i_{k+1}i_k}+ \cdots + 
\alpha_{i_2i_1}$. Let us emphasize that a path is not a sequence of vertices but edges. If $G$ is a graph, then any path is uniquely defined by the sequence of its vertices, if
$G$ is a multigraph, then there may be many paths corresponding to the same sequence of vertices.
If the path is closed ($i_1 = i_{k+1}$), then $\Pi_{\omega}$ maps the space
$L_{i_1}$ to itself. In this case $\Pi_{\omega}$ is given by a square matrix, and possesses eigenvalues, eigenvectors and a spectral radius $\rho(\Pi_{\omega})$.
The set of all
closed paths will be denoted by $\cC (G)$. For an arbitrary $\omega \in \cC(G)$ we denote by
$\omega^k = \omega \cdots \omega$ the $k$th power of $\omega$. 

In what follows we assume that all the sets 
$\cA_{ji}$ are finite.

Now we recall the concept of multinorm introduced in~\cite{PhJ1} and adopted  to arbitrary multigraph with arbitrary linear spaces in~\cite{CGP}.    
\begin{defi}\label{d.200}
If every  space $L_i$ on the multigraph $G$ is equipped with a norm $\|\cdot \|_i$, then
the collection of norms $\|\cdot \|_i, \,
i = 1, \ldots , n$, is called  a multinorm.
The norm of an operator $A_{ji} \in \cA_{ji}$ is defined as
$\, \|A_{ji}\| \, = \sup\limits_{x \in L_i, \|x\|_i = 1}\|A_{ji}x\|_j$.
\end{defi}
Note that the notation $\|x\|_i$ assumes that $x \in L_i$.
In the sequel we suppose that the multigraph $G$ is equipped with some multinorm
$\{ \|\cdot\|_i\}_{i=1}^n$.
We denote that multinorm by $\|\cdot\|$ and sometimes use the
short notation $\|x\| = \|x\|_i$ for $x \in L_i$, that is, we drop the index of the norm if it is clear
to which space $L_i$ the point $x$ belongs.

For a given $x_0 \in L_i$ and for an infinite path $\omega$ starting at the vertex $i$, we
 consider the {\em trajectory} $\{x_k\}_{k \ge 0}$ of the system along this path. Here
 $x_{k} = \Pi_{\, \omega_{k}}\, x_0$, where $\omega_{k}$ is the prefix of~$\omega$ of length~$k$.

As usual, 
the system $\xi$ is called stable if every its trajectory is bounded. 
It is called asymptotically stable if every trajectory tends to zero as $\, k \to \infty$.

As for unconstrained weighted systems, the asymptotic behaviour  is measured in terms of the 
weighted joint spectral radius, which in
 this case is defined as follows:
 \begin{equation}\label{200.jsr}
  \rho(\xi) \ = \ \lim_{k \to \infty}\, 
 \max_{{\rm length}(\omega) = k} \|\Pi_{\omega}\|^{\, 1/|\omega|}\, .
 \end{equation}
Thus, among all paths on $G$ of length $k$ we take one with the 
maximal value $\|\Pi_{\omega}\|^{\, 1/|\omega|}$, then the limit of this 
value as $k\to \infty$ is the joint spectral radius. This limit always exists, as it can be proved following the same arguments as in Lemma~\ref{lem:existence-lim}.

A system is asymptotically stable precisely when there exists a multinorm 
$\|\cdot \| = \{\|\cdot \|_i\}_{i=1}^n$ decreasing along every trajectory. 
This means that the norms of all operators $A_{ji}$ are strictly less than one. 

The concepts of extremal and invariant multinorms~\cite{PhJ1, CGP} 
are also very similar to the corresponding norms. 
A multinorm $\|\cdot \| = \{\|\cdot \|_i\}_{i=1}^n$ is extremal if for every $i$ and $x \in L_i$, we have
\begin{equation}\label{200.extr}
\max_{A_{ji} \in \cA_{ji}, \, j=1, \ldots , n} \,\rho(\xi)^{\, -\alpha_{ji}}\,  \|A_{ji}x\|_j \ \le \ 
\|x\|_i\,  .
\end{equation}
A multinorm is called invariant if   for every $i=1, \ldots , n$ and $x \in L_i$, we have
\begin{equation}\label{200.invar}
\max_{A_{ji} \in \cA_{ji}, \, j=1, \ldots , n} \, \rho(\xi)^{\,- \alpha_{ji}}\, \|A_{ji}x\|_j  \ = \ 
\|x\|_i\, .
\end{equation}

Thus, up to the normalization where one replaces every $A_{ji}$ by $\rho^{\, -\alpha_{ji}}A_{ji}$,
an invariant multinorm is non-increasing along every trajectory, and for every 
$i$ and for every starting point $x_i \in L_i$, there exists an infinite  
trajectory $x_i = x(0)\to x(1)\to x(2)\to \cdots $ such that 
$\|x(0)\| = \|x(1)\| = \|x(2)\| =  \cdots $.

The existence of invariant and of extremal multinorms was proved in~\cite{CGP} 
under the same assumptions as in Theorem~\ref{th.10}. 
The algorithm constructing extremal 
polytope multinorms (when each norm $\|\cdot \|_i$ in the space~$L_i$ is defined by a convex polytope~$P_i$) was presented in the same paper. 
In examples  and in statistics of numerical experiments it was shown that 
 the algorithm is able to find  precisely the joint spectral radius for a vast majority of constrained systems for reasonable time 
 in dimensions  up to~$20$. For positive systems, it works much faster and is applicable in higher dimensions (several hundreds).  It is interesting that 
 the case of reducible system, when all operators share a common subspace, 
 being very rare for unconstrained systems,   becomes usual, or even generic for 
 system on graphs. That is why a special procedure of reducibility was 
 elaborated in~\cite{CGP}.

\subsection{Mixed systems on graphs}\label{sec:4.2}
%

The constrained systems  or systems on graphs appeared 
almost simultaneously in several works. All of them deal with 
discrete-time systems. 
To the best of our knowledge, there is no reasonable 
concept of a continuous-time system on a graph.  
Indeed, the existence of several spaces (vertices)
between which the system can be transferred can  naturally be realized 
in the discrete-time model, but any extension to continuous time
seems hardly possible. Nevertheless, mixed system   
can be realized on graphs and for them various type of constraints can be 
introduced.  

Let us have an arbitrary (discrete-time) system $(G, \cL, \cA, \balpha)$  
on a multigraph $G$, with the spaces~$\cL = \{L_1, \ldots , L_n\}$, 
operators~$\cA = \{A_{ji}\}$, and their 
weights~$\balpha = \{\alpha_{ji}\}$. Let us in addition have 
a family $\cB = \{\cB_1, \ldots , \cB_n\}$, where each $\cB_i$ is a bounded set of 
operators acting on the space $L_i$. This identifies a mixed (discrete-continuous) system 
$\xi = (G, \cL, \cA, \balpha, \cB)$ on the multigraph~$G$. A trajectory 
of this system along an infinite path $\omega: \, i_1 \to i_2 \to \ldots $
is a solution of the system of equations on the spaces~$L_{i_k}$: 
\begin{equation}\label{eq.g-mix}
\dot x(t) \, = \, B_k(t)x(t) \ , \qquad B_k(t)\in \cB_{i_k}, \ t \in [b_{k}, a_{k+1}], \quad x(b_k) = A_{i_k i_{k-1}}x(a_k), \quad k =  1, 2, \ldots  , 
\end{equation}
where $(a_k, b_k)$ are given time-intervals of lengths~$\alpha(k) \in \balpha$. 
In analogy with Definition~\ref{d.100}, the union of those intervals 
is a dark domain and its complement to $\re_+$ is  an active domain
$\cT$. Thus, for each $k$, we have a measurable function $B_k: 
[b_{k}, a_{k+1}]\to \cB_{i_k}$.  On every space $L_{i_k}$ in  the path 
$\omega$, we have a continuous-time system~(\ref{eq.g-mix}) on a segment with operators from $\cB_{i_k}$ . The solutions of those systems are concatenated by 
edges-operators $A_{ji}$ corresponding to the path~$\omega$. 
Such a concept of solutions extends to paths of finite length, for which  
continuous-time dynamics are considered on an unbounded interval of the type $[b_k,+\infty)$.

The notions of stability, Lyapunov exponent, extremal and invariant multinorms are 
extended in a direct way from the case of unconstrained mixed systems introduced in Section~\ref{sec:3}. An analogue of 
Theorem~\ref{th.40} on the existence of the corresponding 
multinorms is formulated and proved in the same way. In particular, 
a mixed system on a graph is asymptotically stable if and only if 
there exists a multinorm $\{\|\cdot \|_i\}_{i=1}^n$ 
in which norms of all operators $A_{ji}$ are smaller than one and
for every $k$ and every $x\in L_k, \|x\|_k = 1$, 
all the vectors $Bx , \, B\in \cB_k$, starting at $x$
are directed inside the unit $\|\cdot\|_k$-sphere.

The algorithm of construction of the extremal polytope Lyapunov norm
is also similar to that for unconstrained mixed systems.

\section{Linear switching systems with guaranteed dwell times}\label{sec:5}
%
%

Now we are able to tackle the main problem: to analyse the stability 
of continuous-time linear switching systems with guaranteed mode-dependent dwell times. 
We are going to see that this can be seen as a special case of a mixed system on a 
graph. In particular, its Lyapunov exponent can be approximately computed 
by constructing a polytope  Lyapunov multinorm.   

Let $\dot x = B(t)x, \ B(t) \in \cB$ be a continuous-time linear switching 
system with finite set of modes 
$\cB = \{B_1, \ldots , B_m\}$.  
Suppose that the dwell time of each operator 
$B_k$ is bounded below by a given number $\alpha_k > 0$. 
This means that the set $\{t \ge 0\,  | \, B(t) = B_k\}$, up to a subset of measure zero, consists of intervals of lengths at least $\alpha_k$. This is a linear switching system with guaranteed dwell times, and it will be denoted by 
$[\cB, \balpha]$, where $\balpha = (\alpha_1, \ldots , \alpha_m)$. 

Denote $A_k = e^{\, \alpha_k B_k}$ and $\cA = \{A_1, \ldots , A_m\}$. 
Every switching law of the system $[\cB, \balpha]$ has a discrete set of 
switching points $\{t_i\}_{i \in D}$, where $D$ is either 
$\{1, \ldots , n\}$ or $\n$. Set $t_0 = 0$. Each segment $[t_{i-1}, t_i]\, , \, i\in D$, corresponds to some operator $B_{k_i} \in \cB$ and has length at least 
$\alpha_{k_i}$. Hence the action of the operator $B_{k_i}$ on this segment can be presented as the action of $A_{k_i}$ followed by the action of 
$B_{k_i}$ on the segment $[t_{i-1} + \alpha_{k_i}\, , \, t_{i}]$. 
We obtain a mixed system with discrete part $\cA$ and continuous part $\cB$. 
This system is constrained: the action of an operator $A_k \in \cA$
is followed by a continuous-time trajectory $\dot x = B_k x$
on some segment (possibly empty), which, in turn, is followed by 
the next mode from $\cA$, etc.

Therefore, this is a
mixed system on a directed strongly connected 
graph 
without loops 
 $G = \{g_1, \ldots , g_m\}$, 
where each space $\cL_k$ is equal to $\re^d$ and all incoming 
edges of the vertex $g_k$ are associated with $A_k$. 
Thus, $\cA_{ji}=\{A_{j}\}$ for $
i,j = 1, \ldots , m$ and $i\ne j$. Each family $\cB_k$
attached with the vertex $g_k$ contains only the operator $B_k$. 

Thus, every vertex $g_k$ has $m-1$ incoming edges 
from the remaining $m-1$ vertices, each of them corresponding to the discrete mode $A_k$, and $m-1$ outgoing edges corresponding to the modes $A_i$, $i\ne k$, with $A_i$ going to $g_i$. 

We can resume the previous remarks by the following statement, where the word {\em isomorphic} is used to express identity of trajectories up to natural identifications.  
\begin{theorem}\label{th.50}
A system with guaranteed dwell times $[\cB, \alpha]$ is 
isomorphic to the mixed system on the graph $G$ defined above. 
\end{theorem}
This enables us to construct an extremal polytope multinorm and to
compute the Lyapunov exponent by the algorithm presented in Section~\ref{sec:4}.

\begin{ex}[Two matrices]  \rm
\label{ex:twomatrices}

Let us consider ${\cB}=\{ B_1, B_2 \}$   with dwell times given by $\alpha_1$ and $\alpha_2$.
We let $\tilde B_1 = e^{\tau B_1}$, $\tilde B_2=e^{\tau B_2}$, $A_1=e^{\alpha_1 B_1}$ and $A_2=e^{\alpha_2 B_2}$. 
The general picture is illustrated by Figure \ref{fig:genpic2}. 
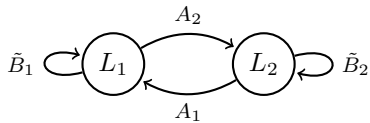
\begin{figure}[ht]
\begin{center}
\begin{tikzpicture}[->,shorten >=1pt,auto,node distance=2cm,
                    thick,main node/.style={circle,draw,font=\sffamily\bfseries}]

  \node[main node] (1) {$L_{1}$};
  \node[main node] (2) [right of=1] {$L_2$};

  \path[every node/.style={font=\sffamily\footnotesize}]
    (1) edge [bend left] node {$A_2$} (2)
        edge [loop left] node [left] {$\tilde B_1$} (1)
    (2) edge [bend left] node {$A_1$} (1)
		    edge [loop right] node [right] {$\tilde B_2$} (2);

\end{tikzpicture}
\caption{Graph associated with Example~\ref{ex:twomatrices} with general $\tau$}\label{fig:genpic2}
\end{center}
\end{figure}

\begin{figure}[ht]
\begin{center}
\begin{tikzpicture}[->,shorten >=1pt,auto,node distance=2cm,
                    thick,main node/.style={circle,draw,font=\sffamily\bfseries}]

  \node[main node] (1) {$L_{1}$};
  \node[main node] (2) [right of=1, node distance = 2.5in] {$L_2$};
	\node[main node] (3) [below right of=1, node distance = 2.0in] {$L_3$};

  \path[every node/.style={font=\sffamily\footnotesize}]
    (1) edge [bend left]  node {$A_2$} (2)
    (1) edge [loop left]  node [left] {$\tilde B_1$} (1)
		(1) edge [bend left]  node [left] {$A_3$} (3)
    (2) edge [bend right = 75]  node {$A_1$} (1)
		(2) edge [loop right] node [right] {$\tilde B_2$} (3)
	  (2) edge [bend left]  node {$A_3$} (3)
		(3) edge [bend left]  node {$A_2$} (2)
		(3) edge [loop below] node [below] {$\tilde B_3$} (3)		
    (3) edge [bend left]  node [left] {$A_1$} (1);    

\end{tikzpicture}
\caption{Graph associated with Example~\ref{example5}
}\label{fig:genpic3}
\end{center}
\end{figure}
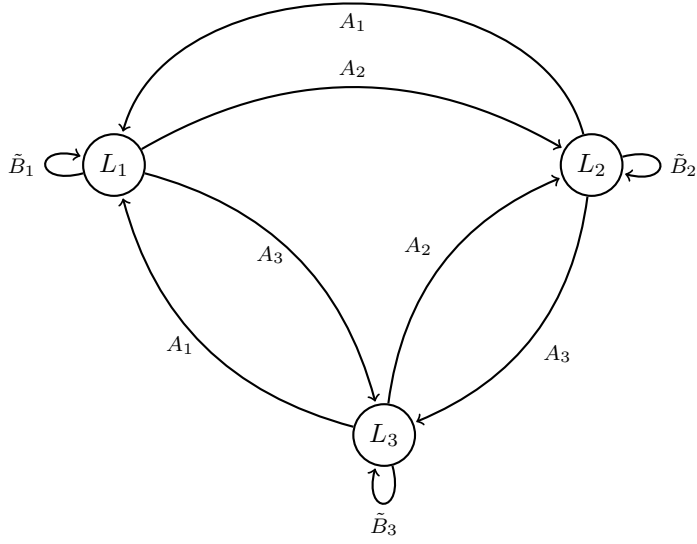


\paragraph*{A numerical example}

Let ${\cB}=\{ B_1, B_2 \}$   with dwell times given by $\alpha_1=1/2$ and $\alpha_2=1$ with
\begin{eqnarray*}
B_1 & = & \left(\begin{array}{rr}
   0 \ & \  0 \\
   1 \ & \  0
\end{array} \right)
\\
B_2 & = & \left(\begin{array}{rr}
   0.60459\ldots &  1.20919\ldots \\
  -1.20919\ldots & -0.60459\ldots
\end{array} \right)
\end{eqnarray*}
which are matrix logarithms,
\begin{eqnarray*}
B_1 = \log \left( \left(\begin{array}{rr}
 1 & 0
\\
 1 & 1
\end{array} \right) \right), \quad
B_2 = \log \left( \left(\begin{array}{rr}
 1 & 1
\\
-1 & 0
\end{array} \right) \right).
\end{eqnarray*}

\begin{itemize}

\item[(i) ] Let us first fix $\tau=1$ and set $\tilde B_1 = e^{\tau B_1}$, $\tilde B_2=e^{\tau B_2}$.

This case is {\em compatible} with $\alpha_1$ and $\alpha_2$ so that there are
no constraints and the problem is a classical unconstrained joint spectral radius
computation for the family $\{ \tilde B_1,\tilde B_2 \}$. We get the following spectrum maximizing product, 
\[
P_1=\tilde B_1^3\,\tilde B_2
\]
which identifies the switching signal that determines the highest growth in the trajectories of the associated
linear system with no constraints, that is the periodic signal $(111211121112\ldots)$, where every value is taken 
on an interval of length $\tau=1$. 

We have $\rho=\rho(P)^{1/4} = 1.389910663524148$, which gives a lower bound for the Lyapunov exponent:
\begin{equation}
\label{eq:firstbound}
\sigma \ge \beta(\tau) = 0.329239474231204.
\end{equation}
Applying Algorithm \ref{algoP2} we obtain that the extremal polytope is invariant for the
shifted vector field $B_{1,2} - (\beta(\tau) + 0.425\ldots) {\rm Id}$, from which we get the upper bound
\begin{equation}
\label{eq:upperbound}
\sigma \le 0.754\ldots .
\end{equation}

\medskip

\item[(ii) ] Let us fix $\tau=2/5$. We consider the approach with $4$ matrices in Figure \ref{fig:genpic2}.
We let $\tilde B_1 = e^{\tau B_1}$, $\tilde B_2=e^{\tau B_2}$, $A_1=e^{\alpha_1 B_1}$ and $A_2=e^{\alpha_2 B_2}$. 

We discover that the following is a spectrum maximizing product,
\[
P=  \tilde B_1^5\,A_1\,  A_2
\]
which identifies the extremal (constrained) periodic signal 
$$(1111122)^k$$ where every value is taken
on an interval of length $1/2$ (see Figure \ref{fig:signals}). 

We have $\rho=\rho(P)^{2/7} = 1.392483264463604$, which gives a lower bound for the 
Lyapunov exponent:
\begin{equation}
\label{eq:secondbound}
\sigma \ge 0.331088674408556 = \beta(\tau),
\end{equation}
improving \eqref{eq:firstbound}.

The polytope algorithm takes $18$ iterations to converge and produces the multinorm in Figure \ref{fig:Ex25_Step5}.

Applying Algorithm \ref{algoP2} we obtain that the extremal polytope is invariant for the
shifted vector field $B_{1,2} - (\beta(\tau) + 0.313\ldots) {\rm Id}$, from which we get the upper bound
\begin{equation}
\label{eq:secondupperbound}
\sigma \le 0.643\ldots 
\end{equation}
which also improves \eqref{eq:upperbound}.

\begin{figure}
\centering
\begin{minipage}[thb]{0.4\linewidth}
  \centering
  \centerline{\includegraphics[width=\linewidth]{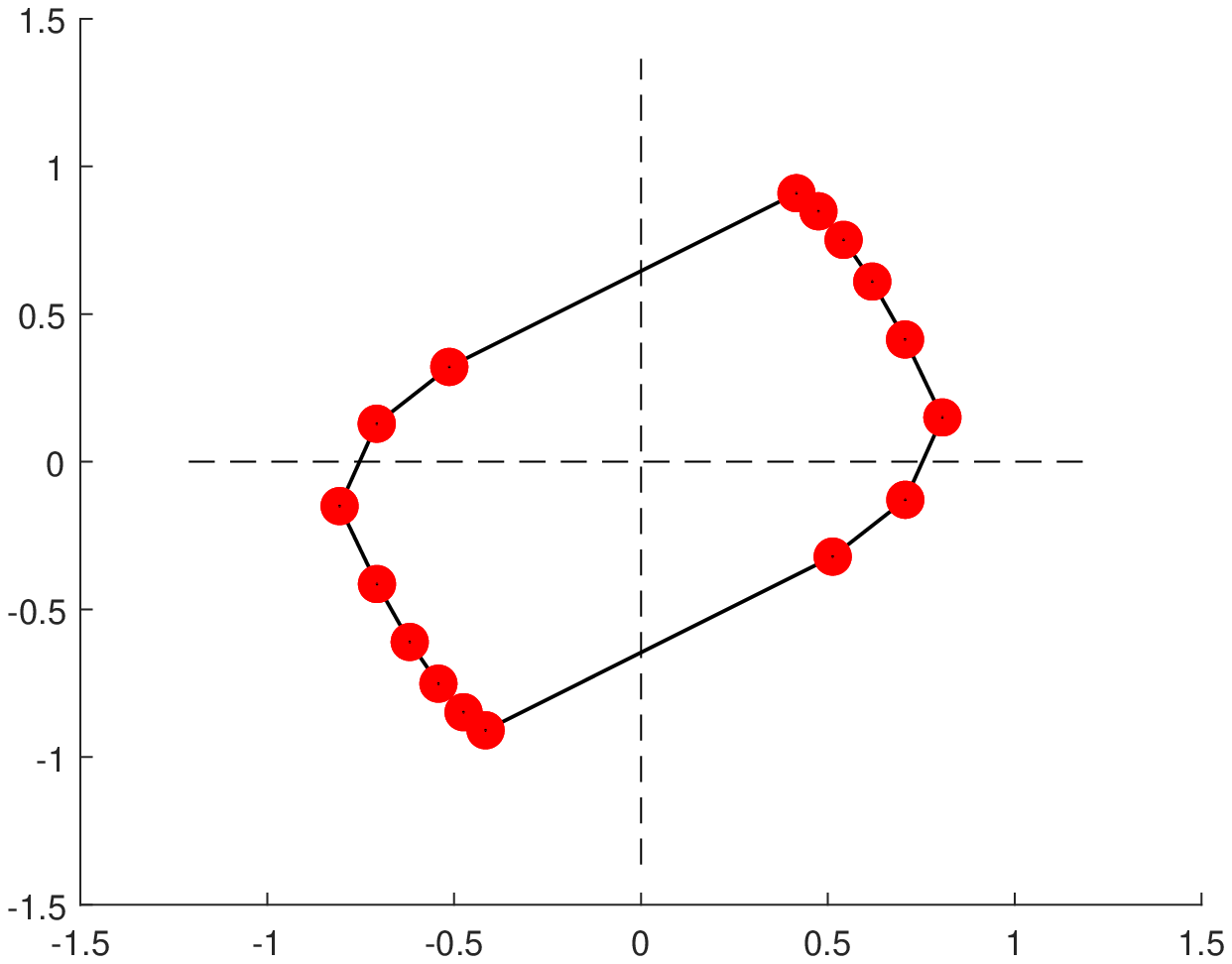}}
  \vspace{1.5cm}
\small\centerline{$L_{1}$}\medskip
\end{minipage}
\hfill
\begin{minipage}[hb]{0.4\linewidth}
  \centering
  \centerline{\includegraphics[width=\linewidth]{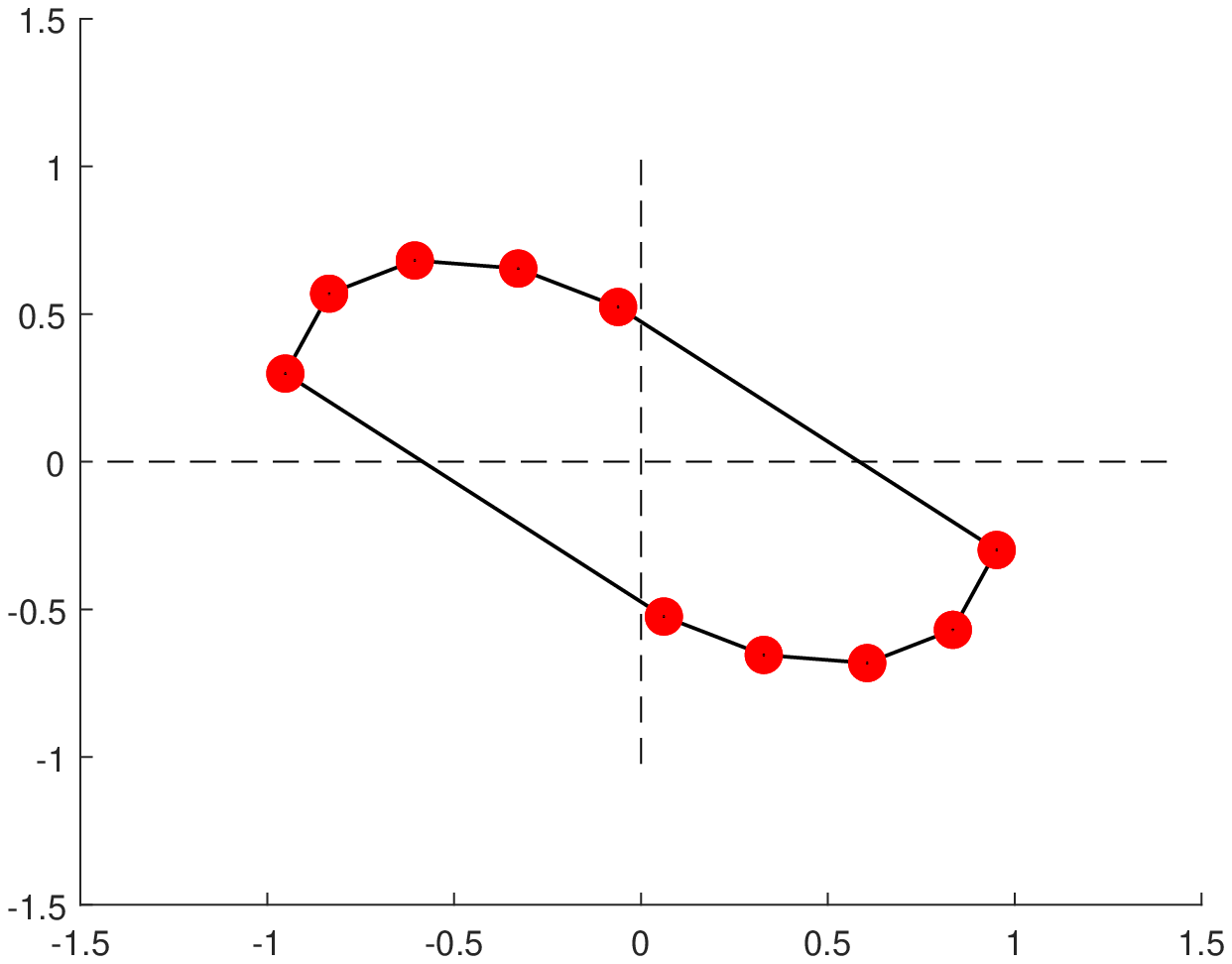}}
  \vspace{1.5cm}
\small\centerline{$L_{2}$}\medskip
\end{minipage}

\vspace*{-0.3cm}
\caption{The polytope extremal multinorm  for Example~\ref{ex:twomatrices} with $\tau=2/5$}\label{fig:Ex25_Step5}

\end{figure}

\medskip

\item[(iii) ] Let us fix $\tau=1/10$. We consider the approach with $4$ matrices in Figure \ref{fig:genpic2}.
We let $\tilde B_1 = e^{\tau B_1}$, $\tilde B_2=e^{\tau B_2}$, $A_1=e^{\alpha_1 B_1}$ and $A_2=e^{\alpha_2 B_2}$. 

We discover that the following is a spectrum maximizing product,
\[
P=  \tilde B_1^{21}\,A_1\,  A_2
\]
which identifies the extremal (constrained) periodic signal 
$$(111111111111111111111111112222222222)^k$$ where every value is taken 
on an interval of length $1/10$ (see Figure \ref{fig:signals}). 

We have $\rho=\rho(P)^{5/36} = 1.392866831588511$, which gives a lower bound for the 
Lyapunov exponent:
\begin{equation}
\label{eq:thirdbound}
\sigma \ge 0.331364091942514 = \beta(\tau),
\end{equation}
improving \eqref{eq:secondbound}.

The polytope algorithm takes $18$ iterations to converge and produces the multinorm in Figure \ref{fig:Ex110_Step18}.

Applying Algorithm \ref{algoP2} we obtain that the extremal polytope is invariant for the
shifted vector field $B_{1,2} - (\beta(\tau) + 0.279\ldots) {\rm Id}$, from which we get the upper bound
\begin{equation}
\label{eq:thirdupperbound}
\sigma \le 0.610\ldots 
\end{equation}
which also improves \eqref{eq:secondupperbound}.

\begin{figure}
\centering
\begin{minipage}[thb]{0.4\linewidth}
  \centering
 \centerline{\includegraphics[width=\linewidth]{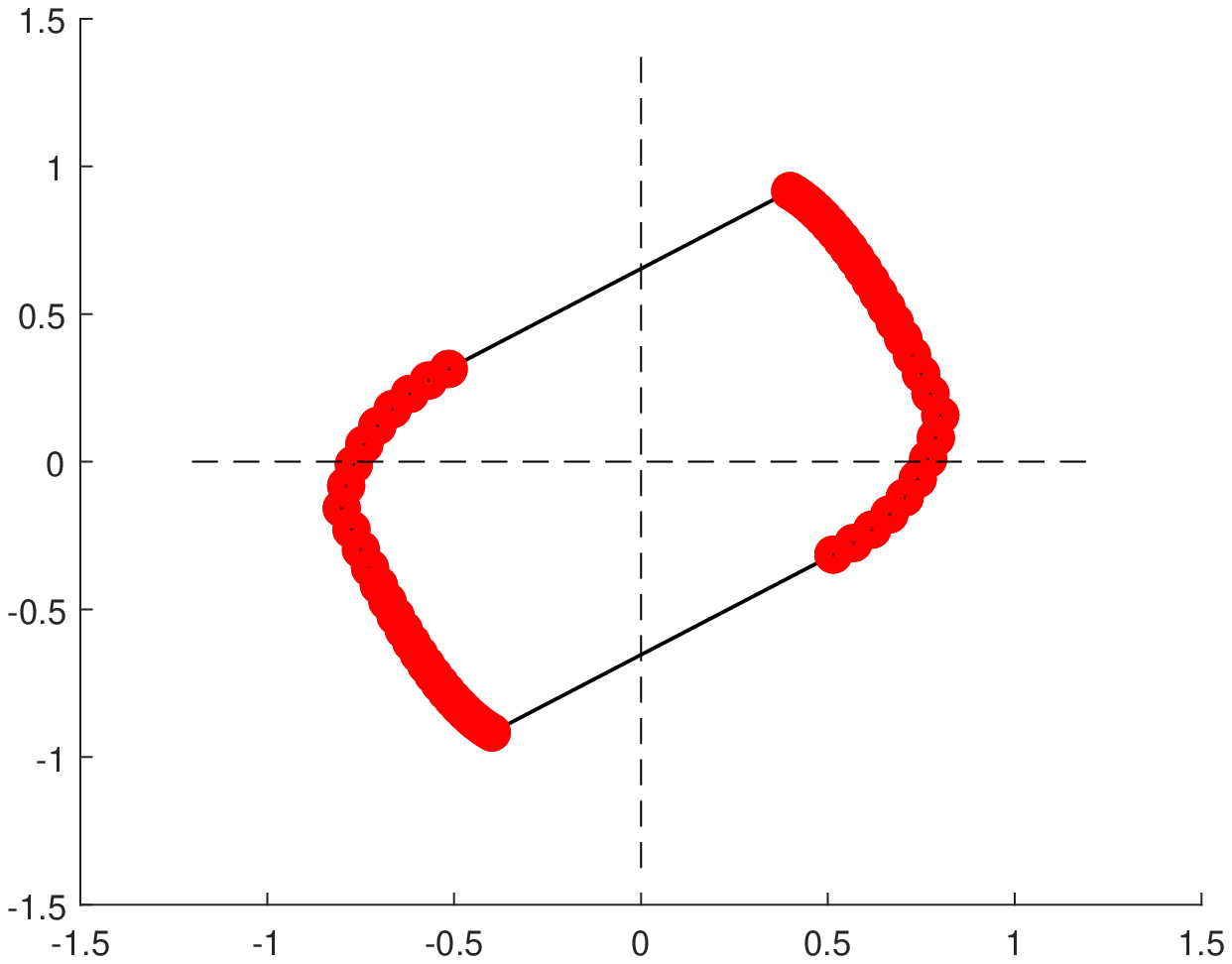}}
  \vspace{1.5cm}
\small\centerline{$L_{1}$}\medskip
\end{minipage}
\hfill
\begin{minipage}[hb]{0.4\linewidth}
  \centering
  \centerline{\includegraphics[width=\linewidth]{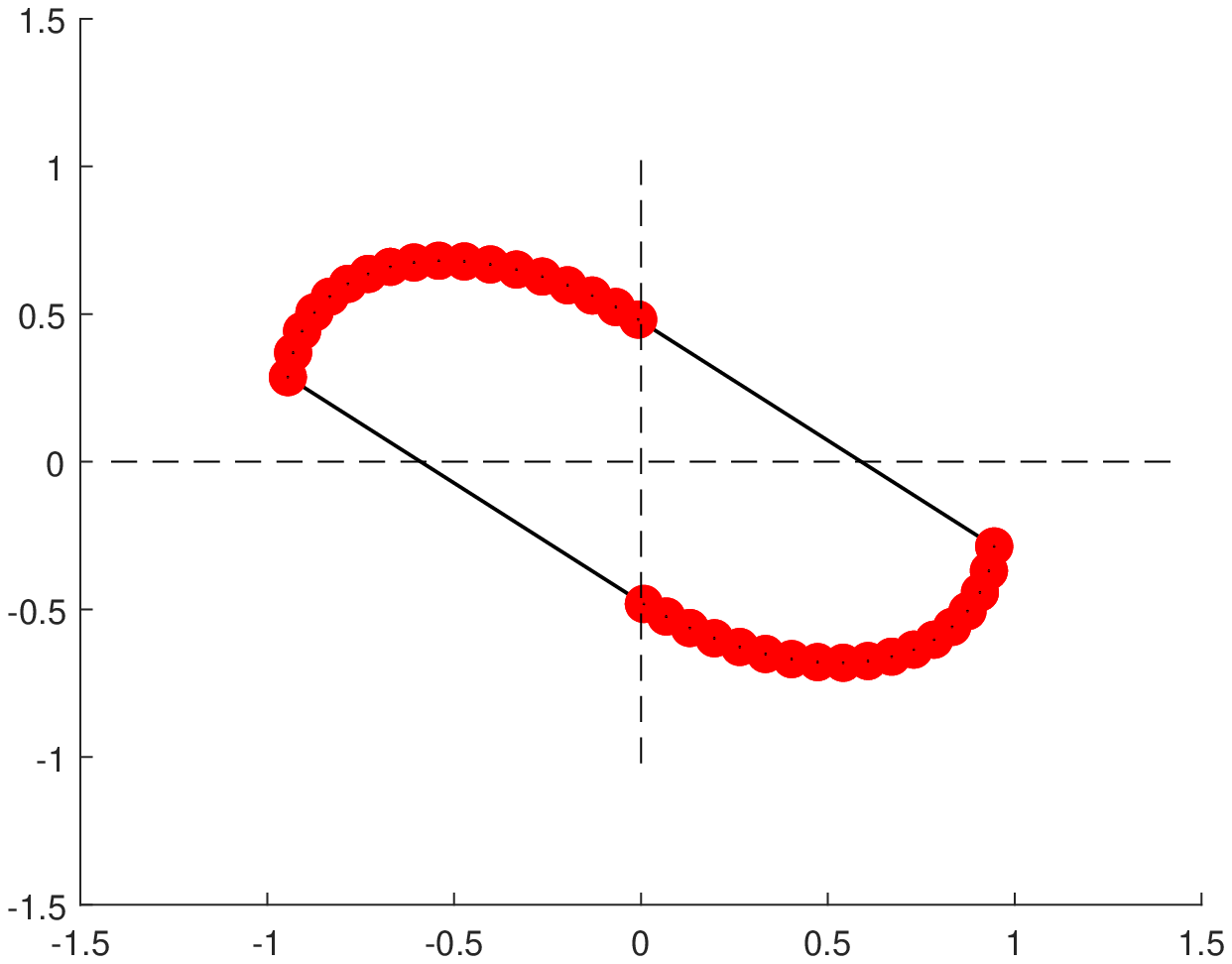}}
  \vspace{1.5cm}
\small\centerline{$L_{2}$}\medskip
\end{minipage}

\vspace*{-0.3cm}
\caption{The polytope extremal multinorm  for Example~\ref{ex:twomatrices} with $\tau=1/10$}\label{fig:Ex110_Step18}

\end{figure}

\begin{figure}[ht]
\begin{center}
\includegraphics[width=0.4\textwidth]{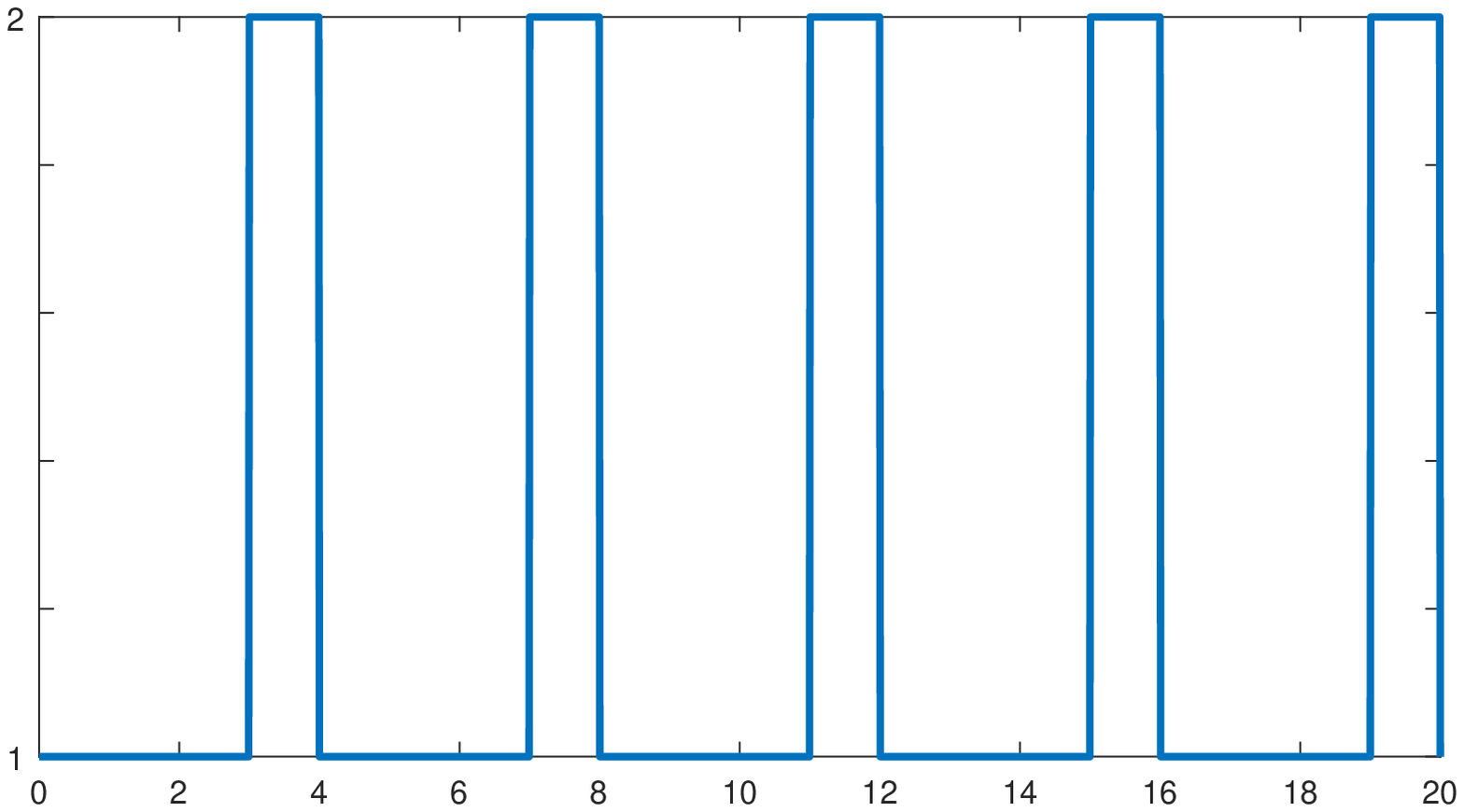} \hskip 1cm \includegraphics[width=0.4\textwidth]{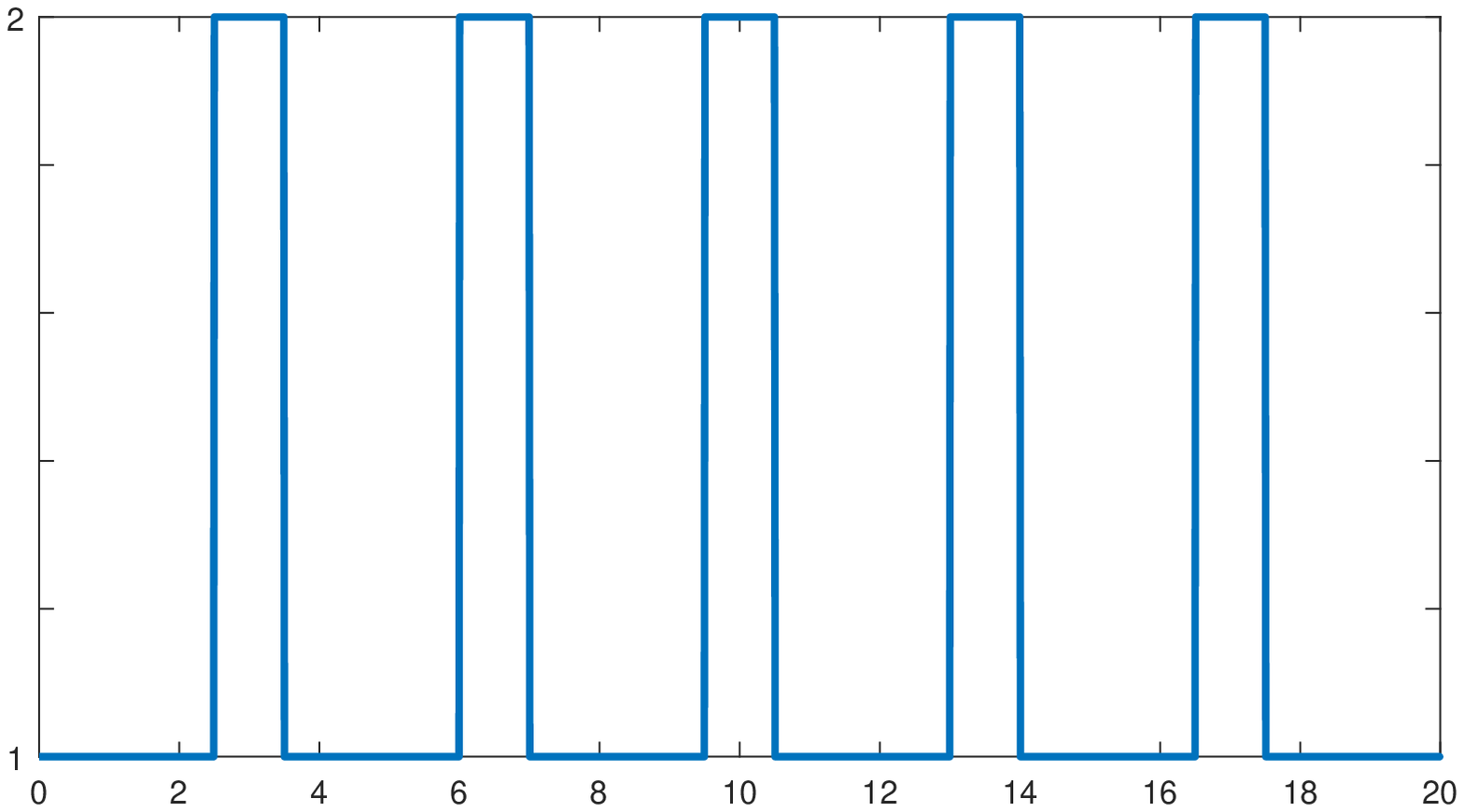} 
\caption{Optimal signals computed for $\tau=1$ (left) and $\tau=1/2$ and $\tau=2/5$ (right). They both 
respect 
dwell times constraints; the signal on the right corresponds to a higher rate of growth. \label{fig:signals}}
\end{center}
\end{figure}

\end{itemize}

\end{ex}

\begin{ex}[Three matrices] \label{example5} \rm

Let us consider ${\cB}=\{ B_1, B_2, B_3 \}$   with dwell times given by $\alpha_1$, $\alpha_2$ and $\alpha_3$.
We let $\tilde B_1 = e^{\tau B_1}$, $\tilde B_2=e^{\tau B_2}$, $\tilde B_3=e^{\tau B_3}$, $A_1=e^{\alpha_1 B_1}$ and $A_2=e^{\alpha_2 B_2}$, $A_3=e^{\alpha_3 B_3}$. 
The general picture is illustrated by Figure \ref{fig:genpic3}. 

\end{ex}

\bigskip

\subsection*{Appendix}

\begin{lemma}\label{lem:fekete-convex}
Let $f:\mathbb{N}\to \mathbb{R}$ be such that for every $j,k\in\mathbb{N}$, $j+k>0$, there exists  $\nu_{j,k}\in [0,1]$ such that 
$f(j+k)\le \nu_{j,k}f(j)+(1-\nu_{j,k})f(k)$ and $\nu_{j,k}\le c \frac{j}{j+k}$ with $c$ independent of $j$ and $k$. 
Then $\lim_{k\to \infty}f(k)$ exists and is equal to $\inf_{k\in \mathbb{N}}f(k)$.
\end{lemma}
\begin{proof} The argument follows the classical proof for sub-additive functions.

Set $f_{\min}:=\inf_{k\in \mathbb{N}}f(k)$ and consider an arbitrary real number $A>f_{\min}$.
Fix a positive integer $m$ 
such that $f(m)< A$. Performing the Euclidean division of every integer $n\geq m$ by $m$ allows one to write
$n=qm+r$ with $q\geq 1$ and $0\leq r<m$. Using the hypothesis on $f$, one has that there exists $\nu_{r,qm}\in [0,1]$ such that
$$
f(n)\leq \nu_{r,qm}f(r)+(1-\nu_{r,qm})f(qm),
$$ 
and $\nu_{r,qm}\leq c\frac{r}{n}$. A trivial induction on $q\ge 1$ yields that  $f(qm)\leq f(m)$. Setting $F(m)=\max_{0\leq r<m}f(r)$, we  deduce that for every $n\geq m$, one has that 
\begin{equation}\label{eq:F0}
f(n)\leq c\frac{m}{n}F(m)+f(m), \quad \forall n\geq m.
\end{equation} 
Letting $n$ tend to infinity, we get that $\limsup_{n\to\infty}f(n)< A$. 
The conclusion follows letting $A$ tend to $f_{\min}$.
\end{proof}

%

\subsection*{Acknowledgments}

The second author (NG) acknowledges financial support from Italian INdAM - GNCS
(Istituto Nazionale di Alta Matematica - Gruppo Nazionale di Calcolo Scientifico).

The work of the fourth author (VP) was prepared within the framework of the HSE 
University Basic Research Program and funded by the Russian Academic Excellence Project 
`5-100'.
Scientific results of Sections 3 and  4 of this paper were obtained with support of RSF Grant 
17-11-01027.

\bibliographystyle{abbrv}


\begin{thebibliography}{10}

\bibitem{Basso}
C.~Basso.
\newblock {\em Switch-Mode Power Supplies Spice Simulations and Practical
  Designs}.
\newblock McGraw-Hill, Inc., New York, NY, USA, 1 edition, 2008.

\bibitem{Girard0}
J.~Ben~Rejeb, I.-C. Mor\u{a}rescu, A.~Girard, and J.~Daafouz.
\newblock Stability analysis of singularly perturbed switched linear systems.
\newblock In {\em Control subject to computational and communication
  constraints}, volume 475 of {\em Lect. Notes Control Inf. Sci.}, pages
  47--61. Springer, Cham, 2018.

\bibitem{BS}
C.~Briat and A.~Seuret.
\newblock Affine characterizations of minimal and mode-dependent dwell-times
  for uncertain linear switched systems.
\newblock {\em IEEE Trans. Automat. Control}, 58(5):1304--1310, 2013.

\bibitem{Chesi1}
G.~Chesi and P.~Colaneri.
\newblock Homogeneous rational {L}yapunov functions for performance analysis of
  switched systems with arbitrary switching and dwell time constraints.
\newblock {\em IEEE Trans. Automat. Control}, 62(10):5124--5137, 2017.

\bibitem{chesi0}
G.~Chesi, P.~Colaneri, J.~C. Geromel, R.~Middleton, and R.~Shorten.
\newblock A nonconservative {LMI} condition for stability of switched systems
  with guaranteed dwell time.
\newblock {\em IEEE Trans. Automat. Control}, 57(5):1297--1302, 2012.

\bibitem{CMS2}
Y.~Chitour, P.~Mason, and M.~Sigalotti.
\newblock A characterization of switched linear control systems with finite
  {$L_2$}-gain.
\newblock {\em IEEE Trans. Automat. Control}, 62(4):1825--1837, 2017.

\bibitem{CGP}
A.~Cicone, N.~Guglielmi, and V.~Y. Protasov.
\newblock Linear switched dynamical systems on graphs.
\newblock {\em Nonlinear Anal. Hybrid Syst.}, 29:165--186, 2018.

\bibitem{D}
X.~Dai.
\newblock Robust periodic stability implies uniform exponential stability of
  {M}arkovian jump linear systems and random linear ordinary differential
  equations.
\newblock {\em J. Franklin Inst.}, 351(5):2910--2937, 2014.

\bibitem{GC}
J.~C. Geromel and P.~Colaneri.
\newblock Stability and stabilization of continuous-time switched linear
  systems.
\newblock {\em SIAM J. Control Optim.}, 45(5):1915--1930, 2006.

\bibitem{Teel00}
R.~Goebel, R.~G. Sanfelice, and A.~R. Teel.
\newblock {\em Hybrid dynamical systems}.
\newblock Princeton University Press, Princeton, NJ, 2012.
\newblock Modeling, stability, and robustness.

\bibitem{G}
G.~Gripenberg.
\newblock Computing the joint spectral radius.
\newblock {\em Linear Algebra Appl.}, 234:43--60, 1996.

\bibitem{GLP15}
N.~Guglielmi, L.~Laglia, and V.~Protasov.
\newblock Polytope {L}yapunov functions for stable and for stabilizable {LSS}.
\newblock {\em Found. Comput. Math.}, 17(2):567--623, 2017.

\bibitem{GP13}
N.~Guglielmi and V.~Protasov.
\newblock Exact computation of joint spectral characteristics of linear
  operators.
\newblock {\em Found. Comput. Math.}, 13(1):37--97, 2013.

\bibitem{GP16}
N.~Guglielmi and V.~Protasov.
\newblock Invariant polytopes of linear operators with applications to
  regularity of wavelets and of subdivisions.
\newblock {\em SIAM J. Matrix Anal. Appl.}, 37(1):18--52, 2016.

\bibitem{GWZ05}
N.~Guglielmi, F.~Wirth, and M.~Zennaro.
\newblock Complex polytope extremality results for families of matrices.
\newblock {\em SIAM J. Matrix Anal. Appl.}, 27(3):721--743, 2005.

\bibitem{HespanhaMorse}
J.~Hespanha and S.~Morse.
\newblock Stability of switched systems with average dwell-time.
\newblock In {\em Proceedings of the 38th IEEE Conference on Decision and
  Control}, 1999.

\bibitem{IngallsSontagWang}
B.~Ingalls, E.~Sontag, and Y.~Wang.
\newblock An infinite-time relaxation theorem for differential inclusions.
\newblock {\em Proceedings of the American Mathematical Society},
  131(2):487--499, 2003.

\bibitem{J09}
R.~Jungers.
\newblock {\em The joint spectral radius}, volume 385 of {\em Lecture Notes in
  Control and Information Sciences}.
\newblock Springer-Verlag, Berlin, 2009.

\bibitem{K}
V.~Kozyakin.
\newblock The {B}erger-{W}ang formula for the {M}arkovian joint spectral
  radius.
\newblock {\em Linear Algebra Appl.}, 448:315--328, 2014.

\bibitem{K2010}
T.~Kr{\"o}ger.
\newblock {\em On-Line Trajectory Generation in Robotic Systems: Basic Concepts
  for Instantaneous Reactions to Unforeseen (Sensor) Events}.
\newblock Springer Berlin Heidelberg, Berlin, Heidelberg, 1 edition, 2010.

\bibitem{Liberzon}
D.~Liberzon.
\newblock {\em Switching in systems and control}.
\newblock Systems \& Control: Foundations \& Applications. Birkh\"{a}user
  Boston, Inc., Boston, MA, 2003.

\bibitem{LiberzonMorse}
D.~Liberzon and A.~S. Morse.
\newblock Basic problems in stability and design of switched systems.
\newblock {\em IIEEE Control Systems Magazine}, 19:59--70, 1999.

\bibitem{Mej}
T.~Mejstrik.
\newblock Improved invariant polytope algorithm and applications.
\newblock preprint arXiv:1812.03080.

\bibitem{MP1}
A.~P. Molchanov and E.~S. Pyatnitski\u{\i}.
\newblock Lyapunov functions that define necessary and sufficient conditions
  for absolute stability of nonlinear nonstationary control systems. {III}.
\newblock {\em Avtomat. i Telemekh.}, (5):38--49, 1986.

\bibitem{MP2}
A.~P. Molchanov and Y.~S. Pyatnitskiy.
\newblock Criteria of asymptotic stability of differential and difference
  inclusions encountered in control theory.
\newblock {\em Systems Control Lett.}, 13(1):59--64, 1989.

\bibitem{Morsed}
A.~S. Morse.
\newblock Supervisory control of families of linear set-point controllers. {I}.
  {E}xact matching.
\newblock {\em IEEE Trans. Automat. Control}, 41(10):1413--1431, 1996.

\bibitem{PhJ1}
M.~Philippe, R.~Essick, G.~E. Dullerud, and R.~M. Jungers.
\newblock Stability of discrete-time switching systems with constrained
  switching sequences.
\newblock {\em Automatica J. IFAC}, 72:242--250, 2016.

\bibitem{PhJ3}
M.~Philippe, G.~Millerioux, and R.~M. Jungers.
\newblock Deciding the boundedness and dead-beat stability of constrained
  switching systems.
\newblock {\em Nonlinear Anal. Hybrid Syst.}, 23:287--299, 2017.

\bibitem{SFS}
M.~Souza, A.~Fioravanti, and R.~Shorten.
\newblock Dwell-time control of continuous-time switched linear systems.
\newblock In {\em Proceedings of the 54th IEEE Conference on Decision and
  Control}, pages 4661--4666, 2015.

\bibitem{Taousser}
F.~Z. Taousser, M.~Defoort, and M.~Djemai.
\newblock Stability analysis of a class of switched linear systems on
  non-uniform time domains.
\newblock {\em Systems Control Lett.}, 74:24--31, 2014.

\bibitem{VdSS}
A.~van~der Schaft and H.~Schumacher.
\newblock {\em An introduction to hybrid dynamical systems}, volume 251 of {\em
  Lecture Notes in Control and Information Sciences}.
\newblock Springer-Verlag London, Ltd., London, 2000.

\bibitem{WRDV}
Y.~Wang, N.~Roohi, G.~Dullerud, and M.~Viswanathan.
\newblock Stability of linear autonomous systems under regular switching
  sequences.
\newblock In {\em Proceedings of the 54th IEEE Conference on Decision and
  Control}, pages 5445--5450, 2015.

\bibitem{Xiang2}
W.~Xiang.
\newblock On equivalence of two stability criteria for continuous-time switched
  systems with dwell time constraint.
\newblock {\em Automatica J. IFAC}, 54:36--40, 2015.

\bibitem{Xiang1}
W.~Xiang.
\newblock Necessary and sufficient condition for stability of switched
  uncertain linear systems under dwell-time constraint.
\newblock {\em IEEE Trans. Automat. Control}, 61(11):3619--3624, 2016.

\bibitem{Xu1}
H.~Xu, K.~L. Teo, and W.~Gui.
\newblock Necessary and sufficient conditions for stability of impulsive
  switched linear systems.
\newblock {\em Discrete Contin. Dyn. Syst. Ser. B}, 16(4):1185--1195, 2011.

\end{thebibliography}
\end{document}